\documentclass{article}

 \usepackage[final, nonatbib]{neurips_2019}

\usepackage[numbers]{natbib}

\usepackage{graphicx}
\usepackage{amsmath,amssymb,amsfonts,amstext,amsthm,mathrsfs}
\usepackage[colorlinks = true]{hyperref}       
\usepackage{url}            
\usepackage{booktabs}      
\usepackage{nicefrac}       
\usepackage{microtype}      
\usepackage{cleveref}
\hypersetup{linkcolor =red, citecolor = blue}
\usepackage{dsfont}
\usepackage{enumitem}

\usepackage{algorithm}
\usepackage{algorithmicx}
\usepackage{algpseudocode}
\usepackage[algo2e]{algorithm2e}

\usepackage{threeparttable}
\usepackage{xparse}
\NewDocumentCommand{\ceil}{s O{} m}{%
	\IfBooleanTF{#1} 
	{$\left\lceil#3\right\rceil$} 
	{#2\lceil#3#2\rceil} 
}

\newtheorem{theorem}{Theorem}
\newtheorem{definition}{Definition}

\newtheorem{lemma}{Lemma}

\newtheorem{coro}{Corollary}
\newtheorem{fact}{Fact}
\newtheorem{assum}{Assumption}

\newcommand{\inner}[2]{\left\langle #1, #2 \right\rangle}
\def\defeq{\mathrel{\mathop:}=}

\usepackage{multirow} 
\usepackage[table,dvipsnames]{xcolor}
\usepackage{colortbl}

\newcommand*{\belowrulesepcolor}[1]{%
	\noalign{%
		\kern-\belowrulesep
		\begingroup
		\color{#1}%
		\hrule height\belowrulesep
		\endgroup
	}%
}
\newcommand*{\aboverulesepcolor}[1]{%
	\noalign{%
		\begingroup
		\color{#1}%
		\hrule height\aboverulesep
		\endgroup
		\kern-\aboverulesep
	}%
}

\usepackage{subcaption}  

\title{SpiderBoost and Momentum: Faster Stochastic Variance Reduction Algorithms}

\vspace{-3mm}
\author{Zhe Wang \\
Department of ECE\\
The Ohio State University\\   
\texttt{wang.10982@osu.edu} \\ 
\And
 Kaiyi Ji  \\
Department of ECE\\
The Ohio State University\\   
\texttt{ji.367@osu.edu} \\ 
\And
Yi Zhou \\
Department of ECE\\
The University of Utah\\   
\texttt{yi.zhou@utah.edu} \\ 
\And
Yingbin Liang  \\
Department of ECE\\
The Ohio State University\\   
\texttt{liang.889@osu.edu} \\
\And
Vahid Tarokh  \\
Department of ECE\\
Duke University\\   
\texttt{vahid.tarokh@duke.edu} \\ 
}
\vspace{-3mm}
\begin{document}

\maketitle

\begin{abstract}
SARAH and SPIDER are two recently developed stochastic variance-reduced algorithms, and SPIDER has been shown to achieve a near-optimal first-order oracle complexity in smooth nonconvex optimization. However, SPIDER uses an accuracy-dependent stepsize that slows down the convergence in practice, and cannot handle objective functions that involve nonsmooth regularizers. In this paper, we propose SpiderBoost as an improved scheme, which allows to use a much larger constant-level stepsize while maintaining the same near-optimal oracle complexity, and can be extended with proximal mapping to handle composite optimization (which is nonsmooth and nonconvex) with provable convergence guarantee. In particular, we show that proximal SpiderBoost achieves an oracle complexity of $\mathcal{O}(\min\{n^{1/2}\epsilon^{-2},\epsilon^{-3}\})$ in composite nonconvex optimization, improving the state-of-the-art result by a factor of $\mathcal{O}(\min\{n^{1/6},\epsilon^{-1/3}\})$. We further develop a novel momentum scheme to accelerate SpiderBoost for composite optimization, which achieves the near-optimal oracle complexity in theory and substantial improvement in experiments. 
%
\end{abstract}
\section{Introduction}

We consider the following finite-sum optimization problem
\begin{align}
\min_{x\in \mathbb{R}^d}	\Psi(x)  \defeq f(x) , ~\text{where}~ f(x) \defeq \frac{1}{n}\sum_{i=1}^{n} f_i(x)  \tag{P}
\end{align}
where the function $f$ denotes the total loss on the training samples and in general is nonconvex. Since large-scale machine learning problems can have very large sample size $n$, the full-batch gradient descent algorithm has high computational complexity. Thus, various stochastic gradient descent (SGD) algorithms have been proposed. For nonconvex optimization, the basic SGD algorithm, which calculates the gradient of one data sample per iteration, has been shown to yield an overall stochastic first-order oracle (SFO) complexity, i.e., gradient complexity, of $\mathcal{O}(\epsilon^{-4})$ \citep{Ghadimi2013} to attain a first-order stationary point $\bar{x}$ that satisfies $\mathbb{E}\|\nabla f(\bar{x})\|  \le \epsilon$. SGD algorithms with diminishing  step-size \citep{Ghadimi2013,Bottou2018} or a sufficiently large batch size \citep{ZhouHybird2018,Ghadimi2016} were also proposed to guarantee their convergence to a stationary point rather than its neighborhood. 
 
Furthermore, various variance reduction methods have been proposed, which construct more accurate stochastic gradient estimators than that of SGD, e.g., SAG \citep{Nicolas2012}, SAGA \citep{Defazio2014} and SVRG \citep{Johnson2013}. In particular, SAGA and SVRG have been shown to yield  an overall SFO complexity of $O(n^{2/3}\epsilon^{-2})$ \citep{Reddi2016b,Allen_Zhu2016}.
 Recently, \citep{Lam2017a,Lam2017b} proposed a variance reduction method called SARAH, where the gradient estimator is sequentially updated in the inner loop to improve the estimation accuracy. In particular, SARAH 
has been shown in \citep{Lam2017b} to achieve an overall $\mathcal{O}(\epsilon^{-4})$ SFO complexity for nonconvex optimization. Another variance reduction method called SPIDER was proposed in \citep{Fang2018}, which uses the same gradient estimator as that of SARAH but adopts a normalized gradient update with a stepsize $\eta=\mathcal{O}(\epsilon/L)$. 
\citep{Fang2018} showed that SPIDER achieves an overall $\mathcal{O}(\min\{ n^{1/2}\epsilon^{-2}, \epsilon^{-3 }\})$ SFO, which was further shown to be optimal in the regime with $n \le \mathcal{O}(\epsilon^{-4})$. 

Though SPIDER is theoretically appealing, three important issues still require further attention. First, SPIDER requires a very restrictive stepsize $\eta=\mathcal{O}(\epsilon/L)$ to guarantee its convergence, which prevents SPIDER from making big progress even if it is possible. Relaxing such a condition appears not easy under its original convergence analysis framework. 
\begin{list}{$\bullet$}{\topsep=0.ex \leftmargin=0.15in \rightmargin=0.in \itemsep =-0.05in}
\item {\em This paper proposes a more practical SpiderBoost algorithm, which allows a much larger stepsize $\eta =\mathcal{O}(1/L)$ than SPIDER 
while retaining the same state-of-the-art complexity order as SPIDER (see \Cref{smooth_comparison} in Suppl). This is due to the new convergence analysis idea that we develop, which analyzes the increments of variables over each entire inner loop  rather than over each inner-loop iteration, and hence yields tighter bound and consequently more relaxed stepsize requirement. 
}
\end{list}

Second, the convergence analysis of SPIDER requires a very small per-iteration increment $\|x_{k+1}-x_k\|=\mathcal{O}(\epsilon/L)$, which is difficult to guarantee if one attempts to generalize it to a proximal algorithm for solving the composite optimization problem (see \Cref{sec_proximal}) that possibly involves nonsmoothness. Hence, generalizing SPIDER to the proximal setting with provable convergence guarantee is challenging.
\begin{list}{$\bullet$}{\topsep=0.ex \leftmargin=0.15in \rightmargin=0.in \itemsep =-0.05in}
\item {\em 
Our SpiderBoost has a natural generalization, i.e., the Prox-SpiderBoost algorithm, which can be applied to solve composite optimization problems. 
We show that Prox-SpiderBoost achieves a SFO  complexity of $\mathcal{O}( n^{1/2}\epsilon^{-2} )$ and a proximal oracle (PO)  complexity of $\mathcal{O}(\epsilon^{-2})$, which improves the existing best results by a factor of $\mathcal{O}( n^{1/6} )$ (see \Cref{comparison_nonsmooth}). 
}
\end{list}

Third, although SPIDER achieves the near-optimal oracle complexity in nonconvex optimization, its practical performance has been found \cite{Lam2017b,Fang2018} to be hardly advantageous over SVRG. Therefore, it is of vital importance to exploit other algorithmic dimensions to further improve the practical performance of SPIDER, and momentum is such a promising perspective. However, the existing analysis of variance-reduced algorithms has been explored for SVRG only in certain {\em convex} scenarios \cite{Nitanda2016,Allen-Zhu2017,Allen-Zhu2018,Shang2018} and under a local gradient dominance geometry in nonconvex optimization \cite{Li2017}. Therefore, it is not even clear whether a certain momentum scheme can be applied to SPIDER and yield the optimal oracle gradient complexity for general nonconvex optimization.
\begin{list}{$\bullet$}{\topsep=0.ex \leftmargin=0.15in \rightmargin=0.in \itemsep =-0.05in}
\item {\em This paper proposes a momentum scheme to accelerate the Prox-SpiderBoost, named Prox-SpiderBoost-M, for composite optimization. We show that Prox-SpiderBoost-M achieves an oracle complexity order of $O(n+\sqrt{n} \epsilon^{-2})$, matching the complexity lower bound for nonconvex optimization. In contrast to the existing analysis for stochastic algorithms with momentum \cite{Ghadimi2016prox} for nonconvex optimization, our proof exploits the martingale structure of the gradient estimator to bound the variance term and its accumulations over the entire optimization path in a tight way under the momentum scheme. 
}
\end{list}

Due to space limitation, we relegate several other results to the supplementary materials, including analysis of Prox-SpiderBoost under non-Euclidean geometry and Polyak-{\L}ojasiewicz condition, and analysis of both Prox-SpiderBoost and Prox-SpiderBoost-M for online nonconvex composite optimization.


\renewcommand{\arraystretch}{0.6} 
\definecolor{LightCyan}{rgb}{0.88,1,1}
\begin{table*}[t] 
	\small
	\centering 
	\caption{Comparison of SFO complexity and PO complexity for composite optimization.} \label{comparison_nonsmooth}
	\vspace{2mm}
\scalebox{0.9}{	
	\begin{threeparttable} 
			\begin{tabular}{clllcll} \toprule
				\multirow{2}{*}{Algorithms}& &\multirow{2}{*}{Stepsize $\eta$} &\multicolumn{2}{c}{Finite-Sum}       &\multicolumn{2}{c}{Finite-Sum/Online\footnotemark} \\ 
				\cmidrule{4-5}   \cmidrule{6-7} 
				&  & &\multicolumn{1}{c}{SFO } &\multicolumn{1}{c}{PO}    &\multicolumn{1}{c}{SFO} &\multicolumn{1}{c}{PO} \\   \midrule
				
				ProxGD  &\citep{Ghadimi2016} &$\mathcal{O}(L^{-1})$  & $\mathcal{O}(n \epsilon^{-2} )$   &  $\mathcal{O}( \epsilon^{-2} )$ &N/A  &N/A \\ \midrule
				
				ProxSGD &\citep{Ghadimi2016} &$\mathcal{O}(L^{-1})$  & N/A   &  N/A &$\mathcal{O}(\epsilon^{-4}  )$   &$\mathcal{O}( \epsilon^{-2} )$ \\ \midrule

				ProxSVRG/SAGA &\citep{Reddi2016} &$\mathcal{O}(L^{-1})$ &  $\mathcal{O}(n+ n^{2/3} \epsilon^{-2}  )$    & $\mathcal{O}( \epsilon^{-2} )$  &N/A &N/A \\ \midrule
				
				Natasha1.5  &\citep{Allen_Zhu2017} &$\mathcal{O}(\epsilon^{2/3}L^{-2/3})$ &N/A &N/A    &  $\mathcal{O}(\epsilon^{-3}+  \epsilon^{-10/3}   )$    & $\mathcal{O}( \epsilon^{-10/3} )$  \\ \midrule
				
				ProxSVRG$^\textbf{+}$ &\citep{Li2018} 
			  &$\mathcal{O}(L^{-1})$&$\mathcal{O}(n + n^{2/3}\epsilon^{-2})$& $\mathcal{O}(\epsilon^{-2})$ &   $\mathcal{O}(  \epsilon^{-10/3}  )$    & $\mathcal{O}( \epsilon^{-2} )$   \\  \toprule
				
				\belowrulesepcolor{LightCyan}   
				\rowcolor{LightCyan}
				Prox-SpiderBoost &(This Work) &$\mathcal{O}(L^{-1})$ & $\mathcal{O}(n+ n^{1/2}\epsilon^{-2})$&  $\mathcal{O}(\epsilon^{-2})$ &    $\mathcal{O}(  \epsilon^{-2} + \epsilon^{-3 }  )$    & $\mathcal{O}( \epsilon^{-2} )$  \\   \aboverulesepcolor{LightCyan}  \bottomrule
		\end{tabular}  	 
		\vspace{2mm}
		\begin{tablenotes}\small
			\item[1] {The online setting refers to the case where the objective function takes the form of the expected value of the loss function over the data distribution. Such a method can also be applied to solve the finite-sum problem, and hence the SFO complexity in the last column is applicable to both the finite-sum and online problems. Thus, for algorithms that have SFO bounds available in both of the last two columns, the minimum between the two bounds provides the best bound for the finite-sum problem.}
		\end{tablenotes}
	\end{threeparttable} 
}
\vspace{-3mm}
\end{table*}

\vspace{-3mm}
\subsection{Related Work}


{\bf Stochastic algorithms for smooth nonconvex optimization:} 
The convergence analysis for SGD was studied in \citep{Ghadimi2016} for smooth nonconvex optimization. SGD with diminishing stepsize and sufficiently large batch size were further studied in \citep{Ghadimi2016,Bottou2018,ZhouHybird2018} to improve the performance. Various variance-reduced algorithms have been proposed and studied, including, e.g., SAG \citep{Nicolas2012}, SAGA \citep{Defazio2014}, SVRG \citep{Johnson2013,Reddi2016b,Allen_Zhu2016}, SCSG \citep{Lei2017}, SNVRG \citep{Zhou2018}, SARAH \cite{Lam2017a,Lam2017b,Nguyen2019withSARAH,Pham2019}, SPIDER \cite{Fang2018}. 
In particular, SPIDER has been shown in \citep{Fang2018} to achieve the oracle complexity lower bound for a certain regime. Such an idea has also been extended for optimization over manifolds in \citep{ZhouPan2018,Zhang2018}, zeroth-order optimization in \cite{pmlr-v97-ji19a}, ADMM in \cite{Huang2019ICML}, zeroth-order ADMM in \cite{huang2019nonconvex}, problem with nonsmooth nonconvex regularizer in \cite{Xu2019}, stochastic composite optimization in \cite{zhang2019stochastic}, noisy gradient descent in \cite{LiZhize2019}, and an adaptive batch size scheme in \cite{ji2019faster}. 
Our study here proposes a SpiderBoost algorithm, which substantially improves the stepsize of SPIDER while retaining the same performance guarantee and performs much faster than SPIDER in practice.

 
{\bf Stochastic algorithms for composite nonconvex optimization:} Proximal SGD has been proposed and studied by \citep{Ghadimi2013,Ghadimi2016prox} to solve composite nonconvex optimization problems. Moreover, variance reduced algorithms such as Prox-SVRG and Prox-SAGA \cite{Reddi2016}, Natasha1.5 \cite{Allen_Zhu2017}, and ProxSVRG$^+$ \cite{Li2018} have also been proposed to further improve the performance. Our study proposes Prox-SpiderBoost, which order-level outperforms all the existing algorithms for composite nonconvex optimization.


\textbf{Momentum schemes for nonconvex optimization:}
For nonconvex optimization, \cite{Ghadimi2016prox} established convergence of SGD with momentum to an $\epsilon$-first-order stationary point with an oracle complexity of $O(\epsilon^{-4})$. The convergence guarantee of SVRG with momentum has been explored under a certain local gradient dominance geometry in nonconvex optimization \cite{Li2017}. Here, we propose Prox-SpiderBoost-M which achieves the complexity lower bound for a certain regime, and practically substantially outperforms existing variance reduced algorithms with momentum.

\vspace{-0.1cm}
\section{SpiderBoost for Nonconvex Optimization}
\subsection{SpiderBoost Algorithm}\label{sec:vanilla_spider_plus}
In this section, we introduce the SpiderBoost algorithm designed for the problem (P). 
In \citep{Lam2017a}, a novel gradient estimator was introduced for reducing the variance. More specifically, consider a certain inner loop $\{x_k\}_{k=0}^{q-1}$. The initialization of the estimator is set to be $v_0 = \nabla f(x_0)$. Then, for each subsequent iteration $k$, an index set $S$ is sampled and the corresponding estimator $v_k$ is constructed as
\begin{align}
v_{k} &= \frac{1}{|S|}\sum_{i \in S} \big[\nabla f_i(x_{k}) - \nabla f_i(x_{k-1}) + v_{k-1} \big]. \label{spider} 
\end{align}
 It can be seen that the estimator in \cref{spider} is constructed iteratively based on the information $x_{k-1}$ and $v_{k-1}$ that are obtained from the previous update. As a comparison, the SVRG estimator \cite{Johnson2013} is constructed based on the information of the initialization of that loop (i.e., replace $x_{k-1}$ and $v_{k-1}$ in \cref{spider} with $x_0$ and $v_0$, respectively). Therefore, the estimator in \cref{spider} utilizes more fresh information and yields more accurate estimation of the full gradient. The estimator in \cref{spider} has been adopted by \citep{Lam2017a,Lam2017b} and \citep{Fang2018} for proposing   SARAH and SPIDER, respectively. 
In specific, SPIDER was shown in \citep{Fang2018} to be optimal in the regime with $n \le \mathcal{O}(\epsilon^{-4})$.

 \begin{figure}
 	\begin{minipage}{0.50\linewidth}
 		\begin{algorithm}[H]
 			\caption{SpiderBoost}\label{alg:spiderboost} 
 			{\bf Input:}  { $  \eta =\frac{1}{2L}$}, $q, K, |S|\in \mathbb{N}$.
 			
 			\For{ $k=0, 1, \ldots, K-1$}{
 				\eIf { $\textrm{mod}(k,q) = 0$ } {
 					Compute $v_{k} = \nabla f(x_k),$
 				}{
 					Draw $|S|$ samples with replacement.\\
 					Compute $v_k$ according to \cref{spider}.\\
 				}
 				{$x_{k+1} = x_k - \eta  v_k$.}
 			}
 			{ {\bf Output:} $x_\xi$, where  $\xi \overset{\text{Unif}}{\sim} \{0,\ldots,K-1\}$.}
 		\end{algorithm} 	
 	\end{minipage} 
 	\begin{minipage}{0.50\linewidth} 
 		\begin{algorithm}[H]
 			\caption{Prox-SpiderBoost}\label{alg:Proximal_SPIDER+} 
 			{\bf Input:}  { $  \eta =\frac{1}{2L}$}, $q, K, |S|\in \mathbb{N}$.
 			
 			\For{ $k=0, 1, \ldots, K-1$}{
 				\eIf { $\textrm{mod}(k,q) = 0$ } {
 					Compute $v_{k} = \nabla f(x_k),$
 				}{
 					Draw $|S|$ samples with replacement.\\
 					Compute $v_k$ according to \cref{spider}.\\
 				}
 				{ $x_{k+1} = \mathrm{prox}_{\eta h}(x_k - \eta v_k)$.}}
 			{ {\bf Output:} $x_\xi$,  where  $\xi \overset{\text{Unif}}{\sim} \{0,\ldots,K-1\}$.}
 		\end{algorithm}
 	\end{minipage} 
 \end{figure} 

Though SPIDER has desired performance in theory, it can run very slowly in practice due to the choice of a conservative stepsize. In specific, SPIDER uses a very small stepsize $\eta =\mathcal{O} (\frac{\epsilon}{L})$ (where $\epsilon$ is the desired accuracy) in normalized gradient descent, which yields small increment per iteration, i.e.,  $\|x_{k+1} -x_k\| = \mathcal{O} (\epsilon)$. By following the analysis of SPIDER, such a stepsize appears to be necessary in order to achieve the desired convergence rate. 

Such a conservative stepsize adopted by SPIDER motivates our design of an improved algorithm named SpiderBoost (see \Cref{alg:spiderboost}), which uses the same estimator \cref{spider} as SARAH and SPIDER, but adopts a much larger stepsize $\eta=\frac{1}{2L}$, as opposed to $\eta=\mathcal{O}(\frac{\epsilon}{L})$ taken by SPIDER. 
Also, SpiderBoost updates the variable via a gradient descent step (same as SARAH), as opposed to the normalized gradient descent step taken by SPIDER. Furthermore, SpiderBoost generates the output variable via a random strategy whereas SPIDER outputs deterministically. Collectively, SpiderBoost can make a considerably larger progress per iteration than SPIDER, especially in the initial optimization phase where the estimated gradient norm $\|v_k\|$ is large, and is still guaranteed to achieve the same desirable convergence rate as SPIDER, as we show in the next subsection. We compare the empirical performance between SPIDER and SpiderBoost in \Cref{sec:exp_spider_spiderboost}.

\vspace{-3mm}
\subsection{Convergence Analysis of SpiderBoost}
In this subsection, we study the convergence rate and complexity of SpiderBoost. In particular, we adopt the following standard assumptions.
\begin{assum}\label{assum: f}
	The objective function in the problem (P) satisfies:
	\vspace{-3mm}
	\begin{enumerate}[leftmargin=*]
		\item The object function $\Psi$ is  bounded below, i.e., $\Psi^*:=\inf_{x\in \mathbb{R}^d} \Psi(x) > -\infty$;
		\item Each gradient $\nabla f_i, i=1,...,n$ is $L$-Lipschitz continuous, i.e., $ \forall  x,y\in \mathbb{R}^d$,
		$\|\nabla f_i (x) - \nabla f_i (y)\| \le L\|x-y\|.$
	\end{enumerate}
\end{assum}
\vspace{-3mm}

\Cref{assum: f} essentially assumes that the smooth objective function has a non-trivial minimum and its gradient is Lipschitz continuous, which are valid and standard conditions in nonconvex optimization.
Then, we obtain the following convergence result for SpiderBoost.
	\begin{theorem}\label{finite_sum_thm}
	Let \Cref{assum: f} hold and apply SpiderBoost in \Cref{alg:spiderboost} to solve the problem (P) with parameters $q = |S| = \sqrt{n}$ and stepsize $\eta  = \frac{1}{2L}$. Then, the corresponding output $x_\xi$ satisfies $\mathbb{E}\|\nabla f(x_\xi)\| \le \epsilon$ provided that  the total number $K$ of iterations satisfies 
	\begin{align*}
		K \ge \mathcal{O}\Big(\frac{L(f(x_0) -  f^* )}{\epsilon^2}\Big).
	\end{align*}
	Moreover, the overall SFO complexity is $\mathcal{O}( {\sqrt{n}}{\epsilon^{-2}} +n)$.
	\end{theorem}
\Cref{finite_sum_thm} shows that the output of SpiderBoost achieves the first-order stationary condition within $\epsilon$ accuracy with a total SFO complexity $\mathcal{O}({\sqrt{n}}{\epsilon^{-2}} +n)$. This matches the  lower bound that one can expect for first-order algorithms in the regime $n \le \mathcal{O}(\epsilon^{-4})$ \citep{Fang2018}. 
As we explain in \Cref{sec:vanilla_spider_plus}, SpiderBoost enhances SPIDER mainly due to the utilization of a large constant stepsize, which yields significant acceleration over SPIDER in practice  as we illustrate in the experiments in \Cref{sec:exp_spider_spiderboost}. 

We note that the analysis of SpiderBoost in \Cref{finite_sum_thm} is very different from that of SPIDER that depends on an $\epsilon$-level stepsize and the normalized gradient descent step to guarantee a constant increment $\|x_{k+1} - x_k\|$ in every iteration. In contrast, SpiderBoost exploits the special structure of gradient estimator and analyzes the algorithm over the entire inner loop rather than over each iteration, and thus yields a better bound.  

\newcommand\numberthis{\addtocounter{equation}{1}\tag{\theequation}}

\section{Prox-SpiderBoost for Nonconvex Composite Optimization} \label{sec_proximal}

In this section, we generalize SpiderBoost to solve the following nonconvex composite problem:
\begin{align*} {\min_{x\in \mathcal{X}}	\Psi(x)  \defeq f(x)  + h(x), \quad f(x) \defeq \frac{1}{n}\sum_{i=1}^{n} f_i(x)}  \tag{Q}
\end{align*}
where the function $f$ is possibly nonconvex, $h$ is a simple convex but possibly nonsmooth regularizer, and  $\mathcal{X}$ is a convex constrained set. To handle the nonsmoothness, we next introduce the proximal mapping which is an effective tool for composite optimization.

\subsection{Preliminaries on Proximal Mapping}
Consider a proper and lower-semicontinuous function $h$ (which can be non-differentiable). We define its proximal mapping at $x\in \mathbb{R}^d$ with parameter $\eta>0$ as
\begin{align*}  
\mathrm{prox}_{\eta h}(x):= \arg\min_{u\in \mathcal{X}}  \Big\{h(u) + \frac{1}{2 \eta}\|u-x\|^2 \Big\}.  
\end{align*}
Such a mapping is well defined and is unique particularly for convex functions. Furthermore, the proximal mapping can be used to generalize the first-order stationary condition of smooth optimization to nonsmooth composite optimization via the following fact.

\begin{fact}\label{fact: grad}
	Let $h$ be a proper and convex function. Define the following notion of generalized gradient
	\begin{align}
		G_\eta(x) := \frac{1}{\eta}\Big(x-\mathrm{prox}_{\eta h}(x-\eta \nabla f(x))\Big).
	\end{align}
	 Then, $x$ is a critical point of $\Psi:=f+h$ (i.e., $0\in \nabla f(x) + \partial h(x)$) if and only if $G_\eta(x) = 0.$
\end{fact}
\Cref{fact: grad} introduces a generalized notion of gradient for composite optimization. To elaborate, consider the case $h\equiv 0$ so that the proximal mapping becomes the identity mapping. Then, the generalized gradient $G_\eta(x)$ reduces to the gradient $\nabla f(x)$ of the unconstrained optimization. Therefore, the $\epsilon$-first-order stationary condition for composite optimization is naturally defined as $\|G_\eta(x)\|\le \epsilon$. 


\subsection{Prox-SpiderBoost and Oracle Complexity}

To generalize to composite optimization, SpiderBoost admits a natural extension Prox-SpiderBoost, whereas SPIDER encounters challenges. The main reason is because SpiderBoost admits a constant stepsize and its convergence guarantee does not have any restriction on the per-iteration increment of the variable. However, the convergence of SPIDER   requires the per-iteration increment of the variable to be at the $\epsilon$-level, which is challenging to satisfy under the nonlinear proximal operator in composite optimization.



The detailed steps of Prox-SpiderBoost (which generalizes SpiderBoost to composite optimization objectives) are described in \Cref{alg:Proximal_SPIDER+}. 
In particular, Prox-SpiderBoost updates the variable via a proximal gradient step to handle the possible nonsmoothness in composite optimization. 


%
%
%
%
%
%
%
%

We next characterize the oracle complexity of Prox-SpiderBoost for achieving the generalized $\epsilon$-first-order stationary condition.

\begin{theorem}\label{p_finite_sum_thm}	
	Let \Cref{assum: f} hold and consider the problem (Q) with $\mathcal{X} = \mathbb{R}^d$. Apply the Prox-SpiderBoost in \Cref{alg:Proximal_SPIDER+}  with parameters $q = |S| = \sqrt{n}$ and $\eta = \frac{1}{2L}$. Then, the corresponding output $x_\xi$ satisfies $\mathbb{E}\|G_{\eta}(x_\xi)\|  \le \epsilon$ provided that the total number $K$ of iterations satisfies 
	\begin{align*}
	K \ge \mathcal{O}\Big(\frac{L(\Psi(x_0) -  \Psi^* )}{\epsilon^2}\Big).
	\end{align*}
	Moreover, the SFO complexity is $\mathcal{O}( {\sqrt{n}}{\epsilon^{-2}} +n)$, and the proximal oracle (PO) complexity is $\mathcal{O}(\epsilon^{-2})$.
\end{theorem}



As a comparison, the SFO complexity $\mathcal{O}( {\sqrt{n}}{\epsilon^{-2}} +n)$ of Prox-SpiderBoost in \Cref{p_finite_sum_thm}  improves the existing complexity result by a factor of $n^{1/6}$ \cite{Li2018}. Furthermore, the complexity lower bound for achieving the $\epsilon$-first-order stationary condition in un-regularized optimization \citep{Fang2018} also serves as a lower bound for composite optimization (by considering the special case $h\equiv 0$). Therefore, the SFO complexity of our Prox-SpiderBoost matches the corresponding complexity lower bound in the regime with $ n \le \mathcal{O}(  \epsilon^{-4})$, and is hence near optimal. 

Moreover, our Prox-SpiderBoost still achieves the state-of-the-art convergence results under other settings such as online optimization, non-Euclidean geometry and Polyak-{\L}ojasiewicz condition. Due to the space limitation, we relegate these results  to \Cref{sec_online}.


\section{Accelerating Prox-SpiderBoost via Momentum}
In this section, we propose a proximal SpiderBoost algorithm that incorporates a momentum scheme (referred to as Prox-SpiderBoost-M) for solving the composite problem (Q), and study its theoretical guarantee as well as the oracle complexity. 

\vspace{-2mm}
\subsection{Algorithm Design}
\vspace{-2mm}
We present the detailed update rule of Prox-SpiderBoost-M in \Cref{alg: ProxSPIDERM}.
\begin{algorithm}
	\caption{Prox-SpiderBoost-M}
	\label{alg: ProxSPIDERM}
	{\bf Input:} $q, K \in \mathbb{N}, \{\lambda_k\}_{k=1}^{K-1}, \{\beta_k\}_{k=1}^{K-1} >0 $,   $y_0 = x_0\in \mathbb{R}^d$, and set $\alpha_{k} = \frac{2}{\ceil[]{k/q}+1}$.
	
	\For{$k=0, 1, \ldots, K-1$}
	{
		 $z_{k} = (1-\alpha_{k+1})y_{k} + \alpha_{k+1} x_{k}$, \\
		\eIf{$\text{mod}(k, q)= 0$}
		{
			set $v_{k} = \nabla f(z_k),$
		}
		{
			Draw $\xi_k$ samples with replacement and compute $v_k$ according to \cref{spider}.
		}
		$x_{k+1} =  \mathrm{prox}_{\lambda_{k} h}\big(x_{k} - \lambda_{k} v_k\big)$, \\
		$y_{k+1} = z_{k} - \frac{\beta_{k}}{\lambda_k}x_k + \frac{\beta_{k}}{\lambda_k}  \mathrm{prox}_{\lambda_{k} h}\big(x_{k} - \lambda_{k} v_k\big)$.
	}
	{\textbf{Output:} $z_\zeta$, where $\zeta \overset{\text{Unif}}{\sim} \{0,\ldots,K-1\}$.}
\end{algorithm}

To elaborate on the algorithm design, note that Prox-SpiderBoost-M generates a tuple of variable sequences $\{x_k,y_k,z_k\}_k$ according to the momentum scheme.
In specific, the variables $x_{k}$, $y_{k}$ are updated via proximal gradient-like steps using the gradient estimate $v_k$ proposed for SARAH in \cite{Lam2017a,Lam2017b} and different stepsizes $\lambda_k, \beta_{k}$, respectively. Then, their convex combination with momentum coefficient $\alpha_{k+1}$ yields the variable $z_{k+1}$. Here, we choose a standard momentum coefficient scheduling that diminishes epochwisely (see the expression for $\alpha_{k}$) for proving convergence guarantee in nonconvex optimization. We also note that the two updates for $x_{k+1}$ and $y_{k+1}$ do not introduce extra computation overhead as compared to a single update, since they both depend on the same proximal term.

We want to highlight the difference between our momentum scheme for Prox-SpiderBoost-M and the existing momentum scheme design for proximal SGD in \cite{Ghadimi2016prox} and proximal SVRG in \cite{Allen-Zhu2017}. In these works, they use the following proximal gradient steps for updating the variables $x_{k+1}$ and $y_{k+1}$: 
\begin{align}
x_{k+1} &= \mathrm{prox}_{\lambda_k h}\big(x_{k} - \lambda_{k} v_k\big), \qquad y_{k+1} = \mathrm{prox}_{\beta_k h}\big(z_{k} - \beta_{k} v_k\big). \label{eq: 26}
\end{align}
Note that \cref{eq: 26} use different proximal updates that are based on $x_k$ and $z_k$, respectively. As a comparison, our momentum scheme in \Cref{alg: ProxSPIDERM} applies the same proximal gradient term $\mathrm{prox}_{\lambda_{k} h}\big(x_{k} - \lambda_{k} v_k\big)$ to update both variables $x_{k+1}$ and $y_{k+1}$, and therefore requires less computation. Moreover, our update for the variable $y_{k+1}$ is not a single proximal gradient update (as opposed to \cref{eq: 26}), and it couples with the variables $z_k$ and $x_k$. 

The momentum scheme introduced in \cite{Allen-Zhu2017} was not proven to have a convergence guarantee in nonconvex optimization. In the next subsection, we prove that our momentum scheme in \Cref{alg: ProxSPIDERM} has a provable convergence guarantee for nonconvex composite optimization with convex regularizers. 

\vspace{-3mm}
\subsection{Convergence and Complexity Analysis}
In this subsection, we study the convergence guarantee of Prox-SpiderBoost-M for solving the problem (Q). We obtain the following main result.
\begin{theorem}\label{thm: ProxSpiderM}
	Let \Cref{assum: f} hold. Apply Prox-SpiderBoost-M (see \Cref{alg: ProxSPIDERM}) to solve the problem (Q) with parameters $q=|\xi_k|\equiv \sqrt{n}$, $\beta_k \equiv \frac{1}{8L}$ and $\lambda_k \in [\beta_k, (1+\alpha_k)\beta_{k}]$. Then, the output $z_{\zeta}$ produced by the algorithm satisfies $\mathbb{E}\|G_{\lambda_{\zeta}}(z_{\zeta}, \nabla f(z_{\zeta}))\| \le \epsilon$ for any $\epsilon>0$ provided that the total number  $K$ of iterations satisfies
	\begin{align}
	K \ge \mathcal{O}\bigg(\frac{L(\Psi(x_0)-\Psi^*)}{\epsilon^2} \bigg).
	\end{align}
	Moreover, the SFO complexity is at most $\mathcal{O}(n+\sqrt{n}\epsilon^{-2})$ and the PO complexity is at most $\mathcal{O}(\epsilon^{-2})$.
\end{theorem}

\Cref{thm: ProxSpiderM} establishes the convergence rate of Prox-SpiderBoost-M to satisfy the generalized first-order stationary condition and the corresponding oracle complexity. Specifically, the iteration complexity to achieve the generalized $\epsilon$-first-order stationary condition is in the order of $\mathcal{O}(\epsilon^{-2})$, which matches that of Prox-SpiderBoost. Furthermore, the corresponding SFO complexity $\mathcal{O}(n+\sqrt{n}\epsilon^{-2})$ matches the lower bound for nonconvex optimization \cite{Fang2018}. Therefore, Prox-SpiderBoost-M enjoys the same optimal convergence guarantee as that for the Prox-SpiderBoost  in nonconvex optimization, and it further benefits from the momentum scheme that can lead to significant acceleration in practical applications (as we demonstrate via experiments in \Cref{sec: exp}). 

From a technical perspective, we highlight the following three major new developments in the proof of \Cref{thm: ProxSpiderM} that is different from the proof for the basic stochastic gradient algorithm with momentum \cite{Ghadimi2016prox} for nonconvex optimization: 1) our proof exploits the martingale structure of the SPIDER estimate $v_k$ which allows to bound the mean-square error term $\mathbb{E}\|\nabla f(z_{k}) - v_k\|^2$ in a tight way under the momentum scheme. In traditional analysis of stochastic algorithms with momentum \cite{Ghadimi2016prox}, such an error term corresponds to the variance of the stochastic estimator and is assumed to be bounded by a universal constant. 2): Our proof requires a very careful manipulation of the bounding strategy to handle the accumulation of the mean-square error $\mathbb{E}\|\nabla f(z_{k}) - v_k\|^2$ over the entire optimization path. 

\section{Experiments}\label{sec: exp}

\subsection{Comparison between SpiderBoost and SPIDER}\label{sec:exp_spider_spiderboost}

\begin{figure}[ht]  
	\vspace{-0.3cm}
	\centering 
	\begin{subfigure}{.24\textwidth}
		\centering
		{\includegraphics[width=1\linewidth]{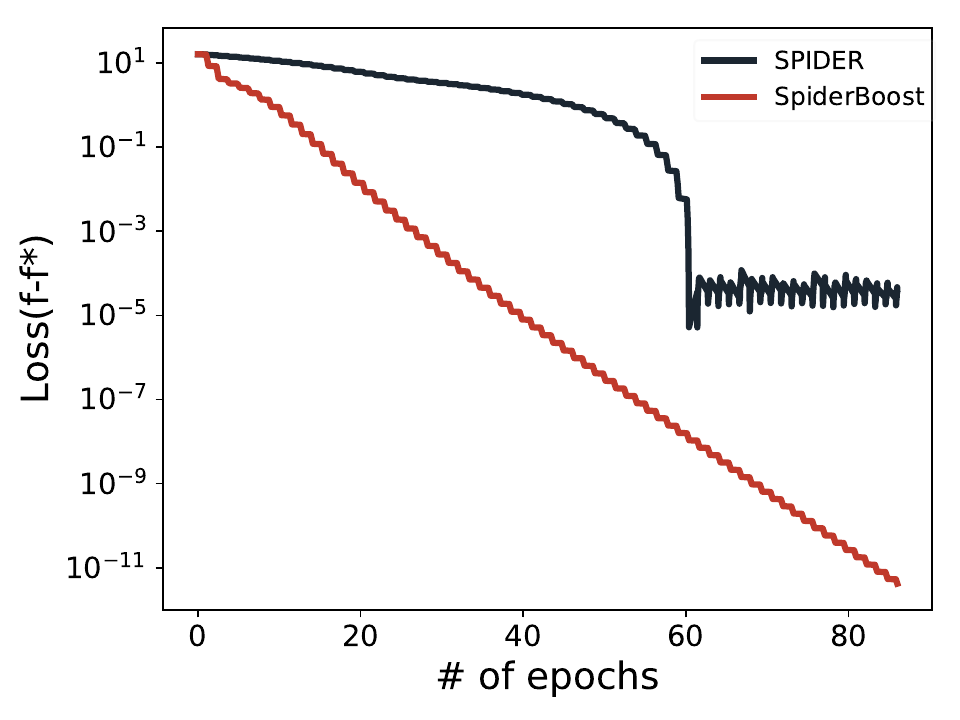}}  
		\caption{Dataset: a9a}
	\end{subfigure} 
	\begin{subfigure} {.24\textwidth}
		\centering
		{\includegraphics[width=1\linewidth]{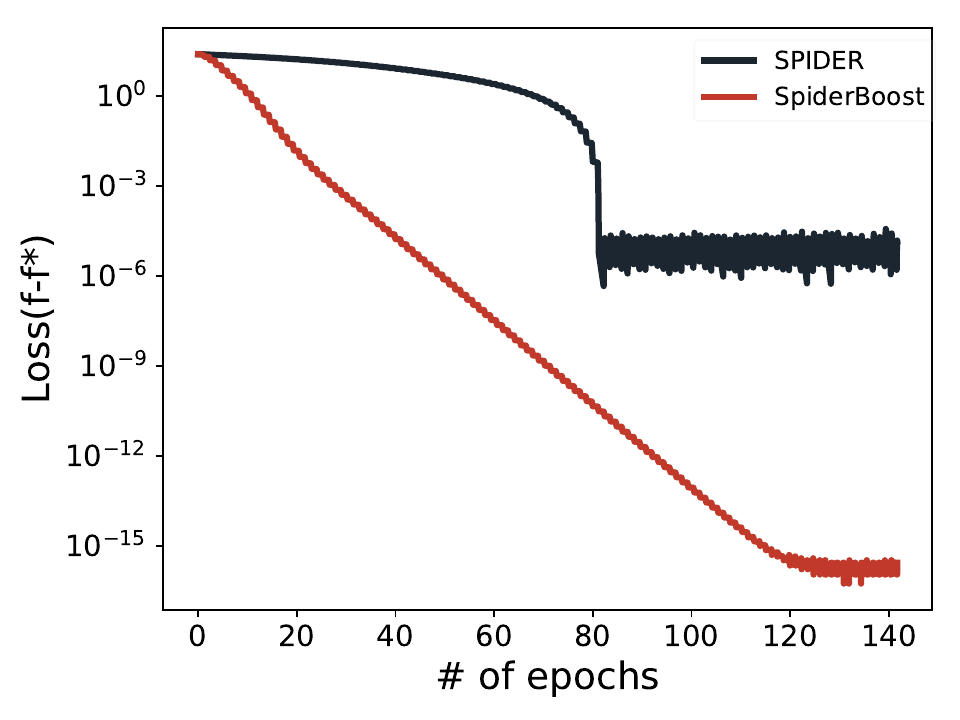}}  
		\caption{Dataset: w8a}
	\end{subfigure}  
	\begin{subfigure}{.24\textwidth}
		\centering
		{\includegraphics[width=1\linewidth]{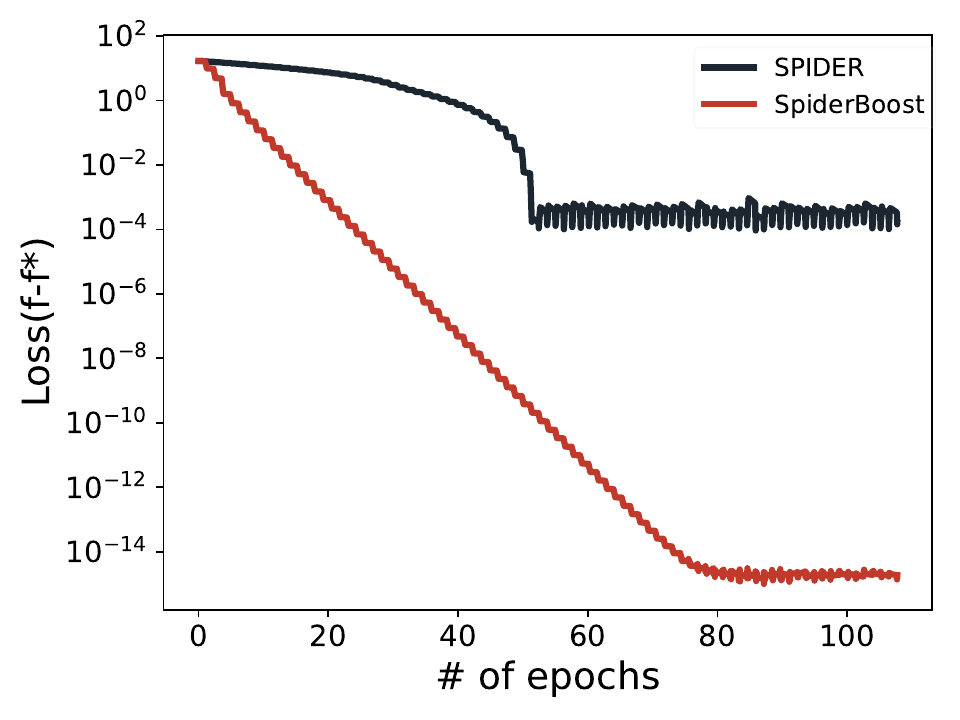}} 
		\caption{Dataset: a9a} 
	\end{subfigure} 
	\begin{subfigure} {.24\textwidth}
		\centering
		{\includegraphics[width=1\linewidth]{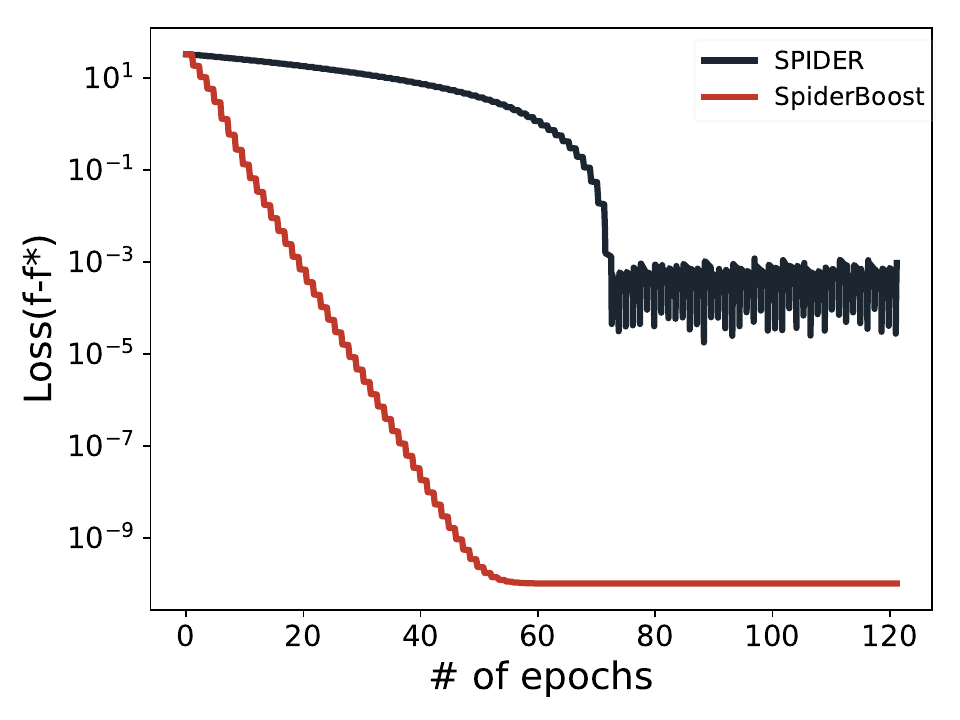}}  
		\caption{Dataset: w8a}
	\end{subfigure}  
	\caption{\small (a) and (b): Logistic regression problem with nonconvex regularizer.  (c) and (d): Robust linear regression problem with an $l_2$ regularizer.}   \label{Experment_1}
	\vspace{-3mm}
\end{figure}
In this subsection, we compare the performance of SPIDER and SpiderBoost for solving the logistic regression problem with a nonconvex regularizer and the nonconvex robust linear regression problem (See \Cref{app:exp} for the forms of the objective functions). For each problem, we apply two different datasets from the LIBSVM \cite{Chang_2011}: the a9a dataset ($n=32561, d=123$) and the w8a dataset ($n=49749, d=300$).  
For both algorithms, we use the same parameter setting except for the stepsize.
As specified in \citep{Fang2018} for SPIDER, we set $\eta = 0.01$ (determined by a prescribed accuracy to guarantee convergence). On the other hand, SpiderBoost allows to set $\eta = 0.05$. \Cref{Experment_1} shows the convergence of the function value gap of both algorithms versus the number of passes that are taken over the data. It can be seen that SpiderBoost enjoys a much faster convergence than that of SPIDER due to the allowance of a large stepsize. Furthermore, SPIDER oscillates around a point, which is the prescribed accuracy that determines the adopted stepsize $\eta = 0.01$. This implies that setting a larger stepsize for SPIDER would cause it to saturate and start to oscillate at a certain function value, which is undesired.

\vspace{-2mm}
\subsection{Comparison of SpiderBoost Type of Algorithms with Other Algorithms}

In this subsection, we compare the performance of our SpiderBoost (for smooth problems), Prox-SpiderBoost (for composite problems), and Prox-SpiderBoost-M  with other existing stochastic variance-reduced algorithms including SVRG in \cite{Johnson2013}, Katyusha$^{ns}$ in \cite{Allen-Zhu2017}, ASVRG in \cite{Shang2018}, RSAG in \cite{Ghadimi2016prox}. We note that all algorithms use certain momentum schemes except for SVRG, SpiderBoost, and Prox-SpiderBoost.
For all algorithms considered, we set their learning rates to be $0.05$. For each experiment, we initialize all the algorithms at the same point that is generated randomly from the normal distribution. Also, we choose a fixed mini-batch size $256$ and set the epoch length $q$ to be $2n/256$ such that all algorithms pass over the entire dataset twice in each epoch. 

We first apply these algorithms to solve two smooth nonconvex problems: logistic regression and robust linear regression problems, each with datasets of a9a and w8a, and report the experiment results in \Cref{Experment_2}. One can see from \Cref{Experment_2} that our Prox-SpiderBoost-M achieves the best performance and significantly outperforms other algorithms. Also, the performances of both Katyusha$^{ns}$ and ASVRG do not achieve much acceleration in such a nonconvex case, as these algorithms are originally developed to achieve acceleration for convex problems. This demonstrates that our design of Prox-SpiderBoost-M has a stable performance in nonconvex optimization as well as provable theoretical guarantee. We note that the curve of SpiderBoost overlaps with that of SVRG similarly to the results reported in other recent studies. 
\begin{figure}[ht]  
	\centering 
	\begin{subfigure}{0.245\linewidth}
		\includegraphics[width=\linewidth]{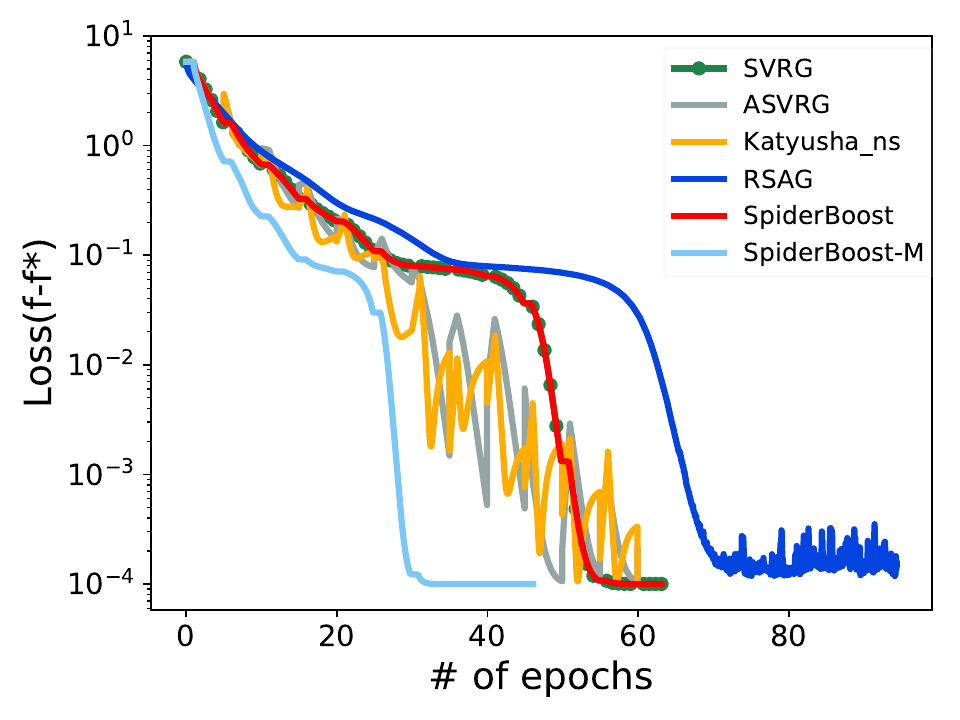}
		\caption{Dataset: a9a}
	\end{subfigure}
	\begin{subfigure}{0.245\linewidth}
		\includegraphics[width=\linewidth]{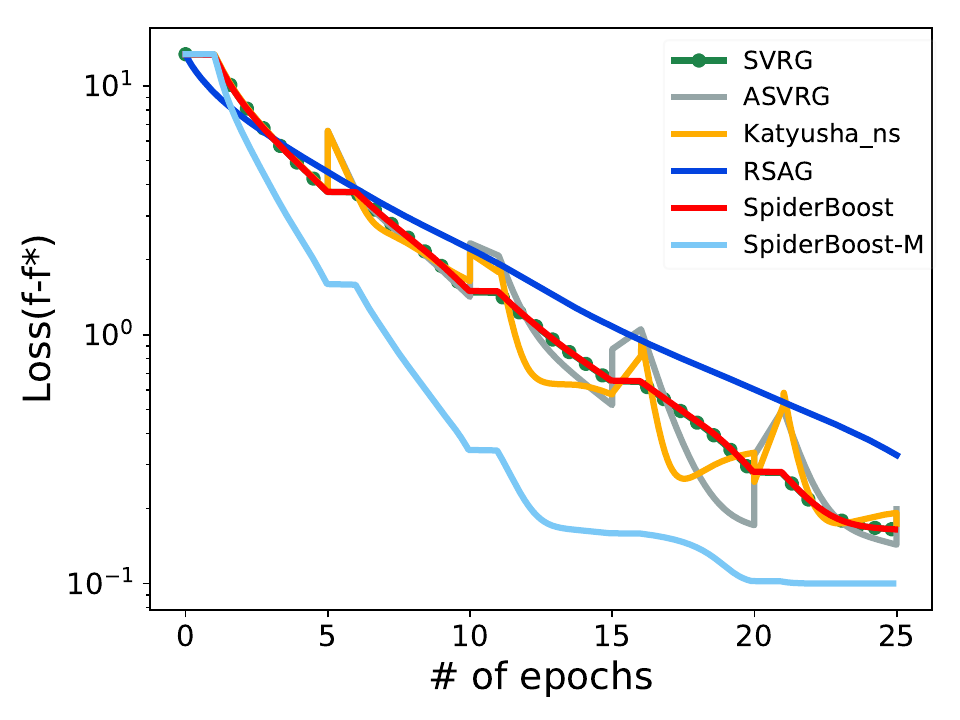}
		\caption{Dataset: w8a}
	\end{subfigure}%
	\begin{subfigure}{0.245\linewidth}
		\includegraphics[width=\linewidth]{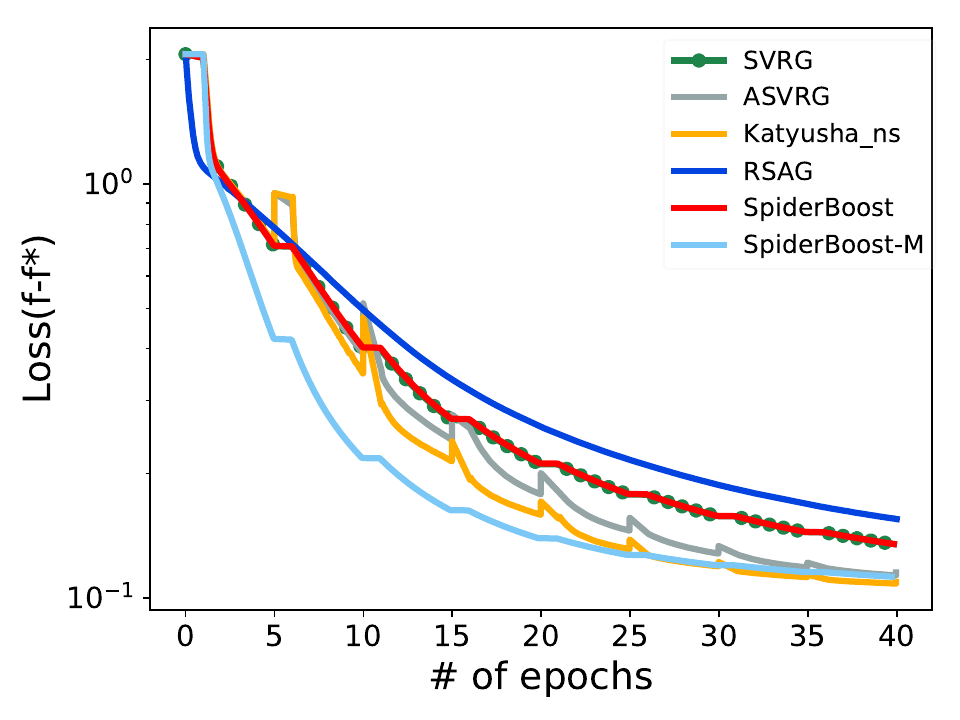}
		\caption{Dataset: a9a}
	\end{subfigure}
	\begin{subfigure}{0.245\linewidth}
		\includegraphics[width=\linewidth]{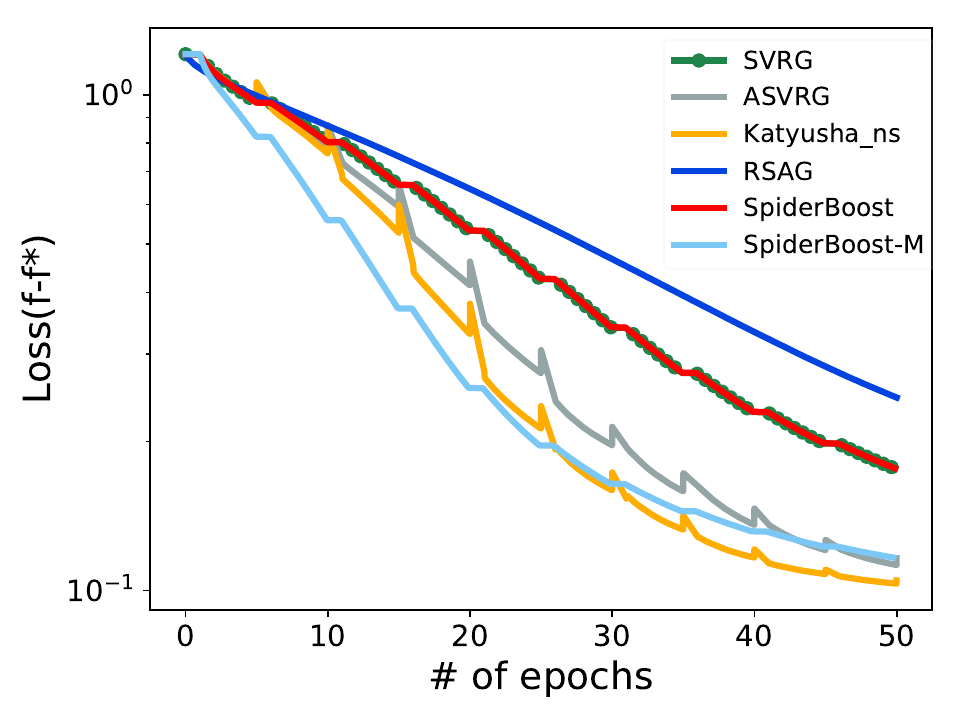}
		\caption{Dataset: w8a}
	\end{subfigure}%
	\caption{\small (a) and (b): Logistic regression with nonconvex regularizer, (c) and (d): Robust linear regression..}   \label{Experment_2}
\end{figure}


We further add an $\ell_1$ nonsmooth regularizer with weight coefficient $0.1$ to the objective functions of the above two optimization problems, and apply the corresponding proximal versions of these algorithms to solve the nonconvex composite optimization problems. All the results are presented in Figures \ref{Experment_3}. One can see that our Prox-SpiderBoost-M still significantly outperforms all the other algorithms in these nonsmooth and nonconvex scenarios. This demonstrates that our novel design of the coupled update for $\{y_k\}_k$ in the momentum scheme is efficient in the nonsmooth and nonconvex setting.
Also, it turns out that Katyusha$^{ns}$ and ASVRG are suffering from a slow convergence (their convergences occur at around 40 epochs). Together with the above experiments for smooth problems, this implies that their performance is not stable and may not be generally suitable for solving nonconvex problems.  
\begin{figure}[ht]  
	\centering
	\begin{subfigure}{0.245\linewidth}
		\includegraphics[width=\linewidth]{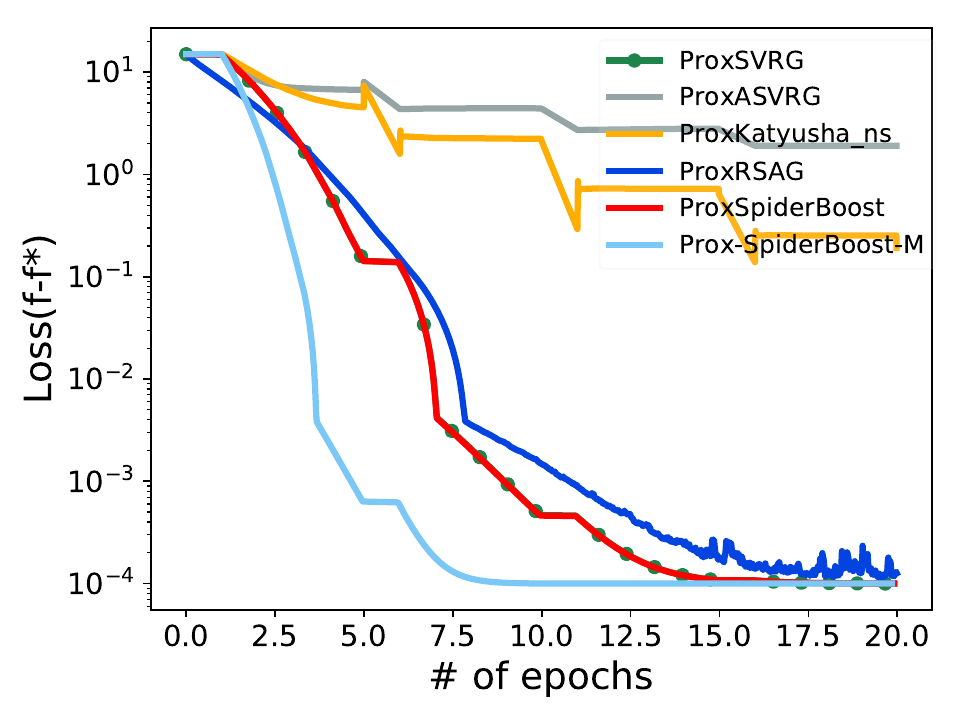}
		\caption{Dataset: a9a}
	\end{subfigure}
	\begin{subfigure}{0.245\linewidth}
		\includegraphics[width=\linewidth]{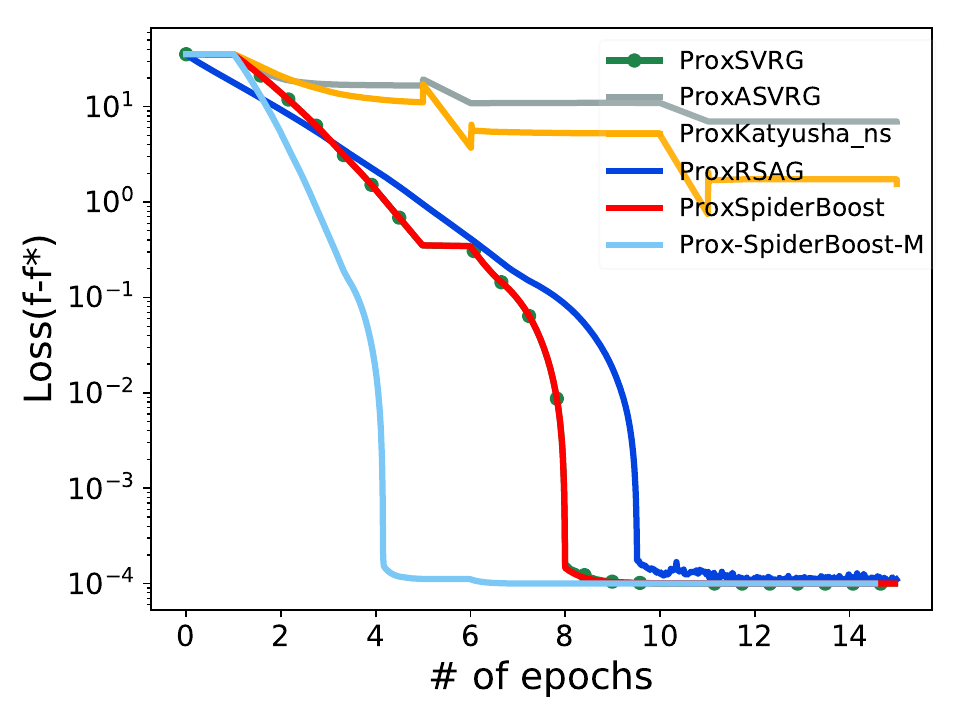}
		\caption{Dataset: w8a}
	\end{subfigure}%
		\begin{subfigure}{0.245\linewidth}
		\includegraphics[width=\linewidth]{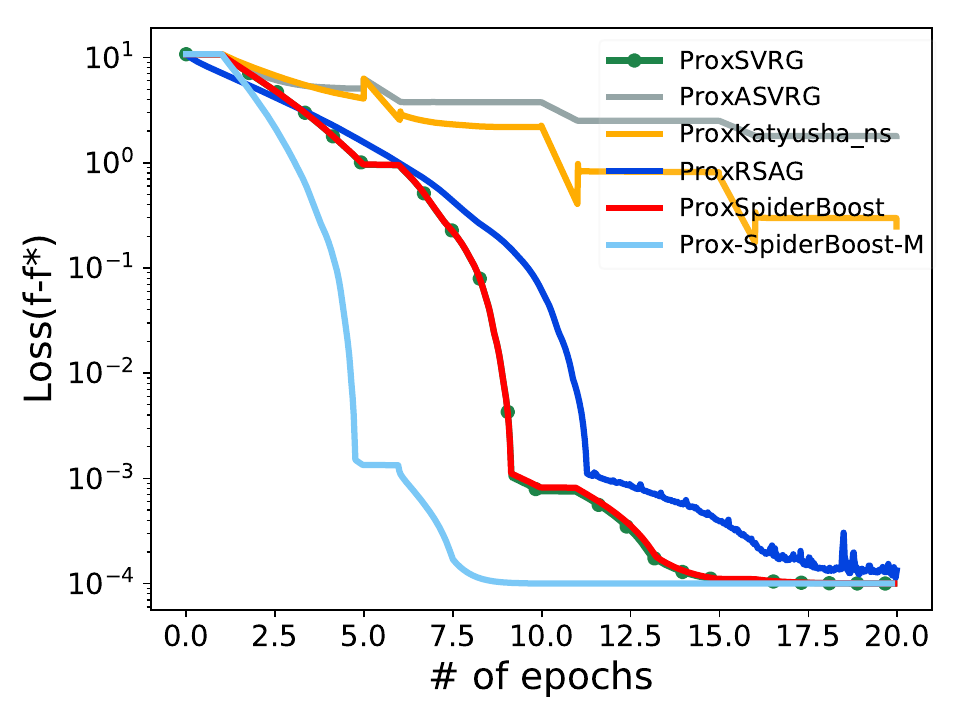}
		\caption{Dataset: a9a}
	\end{subfigure}
	\begin{subfigure}{0.245\linewidth}
		\includegraphics[width=\linewidth]{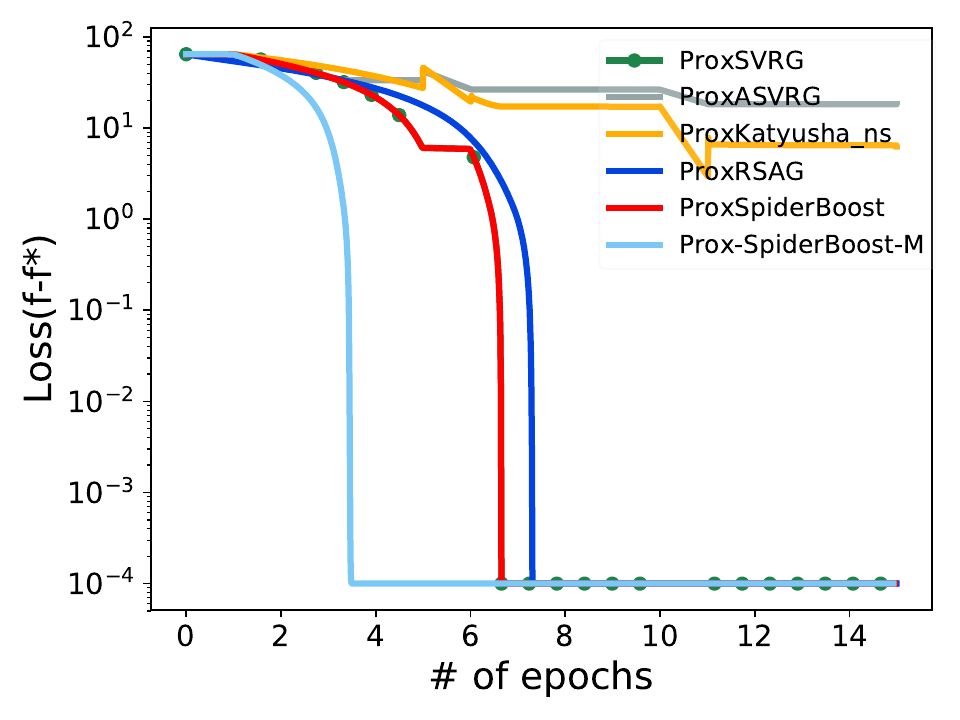}
		\caption{Dataset: w8a}
	\end{subfigure}%
	\caption{\small (a) and (b): Logistic regression with an $\ell_1$ nonsmooth regualarizer. (c) and (d): Robust linear regression with an $\ell_1$ nonsmooth regualarizer.}   \label{Experment_3} 
\end{figure}



\vspace{-3mm}
\section{Conclusion}
\vspace{-2mm}
In this paper, we proposed the SpiderBoost algorithm, which achieves the same near-optimal complexity performance as SPIDER, but allows a much larger stepsize and hence runs faster in practice than SPIDER. We then extend the proposed SpiderBoost to solve composite nonconvex optimization, and proposed a momentum scheme to further accelerate the algorithm. For all these algorithms, we develop new techniques to characterize the performance bounds, all of which achieve the best state-of-the-art. We anticipate that SpiderBoost has a great potential to be applied to various other large-scale optimization problems. 

\section*{Acknowledgments}
The work of Z. Wang, K. Ji, and Y. Liang was supported in part by the U.S. National Science Foundation under the grants CCF-1761506, CCF-1909291, and CCF-1900145.
\bibliography{ref}
\bibliographystyle{abbrv}

\newpage
\onecolumn

\appendix
\noindent {\Large \textbf{Supplementary Materials}}
\section{Comparison of SFO Complexity for Smooth Nonconvex Optimization}
\renewcommand{\arraystretch}{1.3} 
\definecolor{LightCyan}{rgb}{0.88,1,1}
\begin{table*}[h] 
	\small 
	\centering 
	\caption{Comparison of SFO complexity for smooth nonconvex optimization} 
	\label{smooth_comparison}
	\vspace{2mm}
	
	\scalebox{0.9}{
	\begin{threeparttable}
			\begin{tabular}{clllll} \toprule
				\multirow{2}{*}{Algorithms}& &\multirow{2}{*}{Stepsize $\eta$} &\multicolumn{1}{c}{Finite-sum}   &\phantom{a}    &\multicolumn{1}{c}{Finite-sum/Online}  \\  
				&& &\multicolumn{1}{c}{SFO}   & &\multicolumn{1}{c}{SFO}  \\   \midrule
				GD &\citep{Nesterov2014}  &$\mathcal{O}( L^{-1} )$ &  $\mathcal{O}(n \epsilon^{-2} )$  &  &N/A\footnotemark \\  \midrule
				SGD &\citep{Ghadimi2016} &$\mathcal{O}( L^{-1} )$  &  N/A &  &$\mathcal{O}(\epsilon^{-4}  )$    \\  \midrule 
				
				\multirow{2}{*}{SVRG}  &\citep{Reddi2016b}  &\multirow{2}{*}{$\mathcal{O}( L^{-1}n^{-2/3} )$} & \multirow{2}{*}{$\mathcal{O}(n+ n^{2/3}\epsilon^{-2})$} &   &   \multirow{2}{*}{N/A} \\ 
				&\citep{Allen_Zhu2016}  \\ \midrule
				
				SCSG & \citep{Lei2017}  &$\mathcal{O}( L^{-1}(n^{-2/3} \wedge \epsilon^{4/3} ))$ & $\mathcal{O}(n+n^{2/3}\epsilon^{-2})$&   &   $\mathcal{O}(\epsilon^{-2} + \epsilon^{-10/3}  )$  \\  \midrule
				
				SARAH & \citep{Lam2017b,Lam2017a}  &$\mathcal{O}(   (L\sqrt{q})^{-1})$\footnotemark &N/A & &  $\mathcal{O}( \epsilon^{-4})$    \\  \midrule
				
				SNVRG &\citep{Zhou2018}    &$\mathcal{O}(   L^{-1} )$ & $ {\mathcal{O}}( (n+ n^{1/2}\epsilon^{-2}) \log(n))$&   &   $ {\mathcal{O}}(  (\epsilon^{-2}+\epsilon^{-3 }) \log(\epsilon^{-1}) )$  \\ 
				\midrule
				
				SPIDER &\citep{Fang2018}  &$\mathcal{O}( {{  \color{red}{  \epsilon}L^{-1}} } )$\; \footnotemark & $\mathcal{O}(n+ n^{1/2}\epsilon^{-2})$&   &   $\mathcal{O}( \epsilon^{-2} + \epsilon^{-3 }  )$  \\  \midrule \belowrulesepcolor{LightCyan}
				\rowcolor{LightCyan}
				SpiderBoost &(This Work) &$\mathcal{O}( \color{red} L^{-1})$ & $\mathcal{O}(n + n^{1/2}\epsilon^{-2})$&   &   $\mathcal{O}(  \epsilon^{-2} + \epsilon^{-3 }  )$     \\   \aboverulesepcolor{LightCyan}  \bottomrule
		\end{tabular} 
		\vspace{2mm}
		\begin{tablenotes}\small
			\item[2]{For deterministic algorithms, the online setting does not exist.} 
			\item[3]  {The stepsize $\eta = \mathcal{O} (1/(L\sqrt{q}))$ is chosen in \citep{Lam2017a} to guarantee the convergence of SARAH.}
			\item[4] {SPIDER uses the normalized gradient descent, which can also be viewed as the gradient descent with the stepszie $O(\epsilon L^{-1} /\| v_k\|)$.} 
		\end{tablenotes} 
	\end{threeparttable} 
}
\end{table*}

\section{Prox-SpiderBoost for Constrained Optimization under Non-Euclidean Geometry} \label{general_smooth_rlt}
Prox-SpiderBoost proposed in \Cref{sec_proximal} adopts the proximal mapping that solves an unconstrained subproblem under the $\ell_2$ Euclidean distance. Such a mapping can be further generalized to solve constrained composite optimization under a non-Euclidean geometry.  

To elaborate, consider solving the composite optimization problem (Q) subject to a convex constraint set $\mathcal{X}$. We introduce the following Bregman distance $V$ associated with a kernel function $\omega: \mathcal{X} \rightarrow \mathbb{R}$ defined as: for all $x,y \in \mathcal{X}$,
\begin{align}
	V(x,y) = \omega(x) - \omega(y) - \inner{\nabla \omega(y)}{x-y}. \label{p_strongly_convex}
\end{align}
Here, the function $\omega$ is smooth and $\alpha$-strongly convex with respect to a certain generic norm. The specific choice of the kernel function $\omega$ should be compatible to the underlying geometry of the constraint set. For example, for the unconstrained case, one can choose $\omega(x) = \frac{1}{2} \|x \|^2$ so that $V(x,y) = \frac{1}{2} \|x-y\|^2$, which is 1-strongly convex with regard to the $\ell_2$-norm, whereas for the simplex constraint set, one can choose $\omega(x) = \sum_{i=1}^d (x_i \log x_i - x_i)$ that yields the KL relative entropy distance $V(x,y) = \sum_{i=1}^d (x_i \log \frac{x_i}{y_i} + y_i - x_i)$, which is $1$-strongly convex with regard to the $\ell_1$-norm. 
  

Based on the Bregman distance, the proximal gradient step in \Cref{alg:Proximal_SPIDER+} can be generalized to the following update rule for solving the constrained composite optimization.
\begin{align}
T_{\eta h}(x,v) = \arg \min_{u\in \mathcal{X}} \Big\{h(u) + \inner{v}{u} + \frac{1}{\eta}V(u,x) \Big\}. \label{p_d_1}
\end{align}
Moreover, the characterization of critical points in \Cref{fact: grad} remains valid by defining the generalized gradient as $G_{\eta}(x) = \frac{1}{\eta} (x - T_{\eta h}(x,\nabla f(x)))$. Then, we obtain the following oracle complexity result of Prox-SpiderBoost under the Bregman distance (replace the  proximal step in \Cref{alg:Proximal_SPIDER+}  by $x_{k+1} = T_{\eta h}(x_k,v_k)$ ) for solving constrained composite optimization.
\begin{theorem}\label{p_finite_sum_thm_nonEuclidean}
	Let \Cref{assum: f} hold and consider the problem (Q). Apply Prox-SpiderBoost with a proper Bregman distance $V$ that is $\alpha$-strongly convex, where $\alpha > 7/8$. Choose the parameters $q =|S| = \sqrt{n}$ and  $\eta = \frac{1}{2L}$. Then, the algorithm outputs a point $x_\xi$ satisfying $\mathbb{E}\|G_\eta(x_\xi)\|  \le \epsilon$ provided that the total number $K$ of iterations satisfies
	\begin{align*}
	K = \frac{4L(\Psi(x_0) -  \Psi^*)}{\epsilon^2 } \left({\alpha - \frac{7}{8}} \right)^{-1}\left(1 + \frac{1}{4\alpha}\right) .
	\end{align*}
	Moreover, the total SFO complexity is $\mathcal{O}( {\sqrt{n}}{\epsilon^{-2}} +n)$, and the PO complexity is $\mathcal{O}(\epsilon^{-2})$.
\end{theorem}


\section{Prox-SpiderBoost under Polyak-{\L}ojasiewicz Condition}\label{sec:sc}

Despite the nonconvexity geometry,   many machine learning problems have been shown to satisfy the so-called Polyak-{\L}ojasiewicz condition such as phase retrieval \citep{Zhou2016gd}, blind deconvolution \citep{Li_2018} and neural networks \citep{Zhou2017}, etc. This motivates us to explore the theoretical performance of the Prox-SpiderBoost for solving the composite optimization problem (Q) under the generalized Polyak-{\L}ojasiewicz geometry we define below, where the function can still be nonconvex. 
 
\begin{definition}\label{def: sc}
Let $x^*$ be a minimizer of function $\Psi = f+h$. Then, $\Psi$ is said to satisfy the Polyak-{\L}ojasiewicz condition with parameter $\tau$ if for all $x\in \mathbb{R}^d$ and $\eta>0$ one has
\begin{align*}
	\Psi(x) - \Psi(x^*) \le \tau\|G_\eta(x)\|^2,
\end{align*} 
where $G_\eta(x)$ is the generalized gradient defined in \Cref{fact: grad}.
\end{definition}
\Cref{def: sc} generalizes the traditional Polyak-{\L}ojasiewicz condition for single smooth objective functions to composite objective functions. In particular, such a condition allows the objective function to be nonsmooth and nonconvex, and it requires the growth of the function value  to be controlled by the gradient norm.

\begin{algorithm}
	\caption{Prox-SpiderBoost-PL}
	\label{alg:Proximal_SPIDER+_sc}
	
	{\bf Input:} $x_0\in \mathbb{R}^d, q \in \mathbb{N}, \eta<\frac{1}{16L}$.
	
	{\bf \text{For} ~$k=0, 1, 2, \ldots K-1$}
	
	\quad {\bf \text{If}~$\textrm{mod}(k,q) = 0$:} 
	
	\qquad Set $x_k = x_\xi$, where $\xi$ is selected from 
	
	\qquad $\{k-q+1,\ldots,k-2\}$ uniformly at random.
	
	\qquad Compute  $v_{k} = \nabla f (x_k),$
	
	\quad {\bf{\text{Else}:} }
	
	\qquad Draw $|S|$ samples with replacement.
	
	\qquad Compute $v_k$ according to \cref{spider}.
	
	\quad $x_{k+1} = \mathrm{prox}_{\eta h}(x_k - \eta v_k).$

	\textbf{Output:} $x_\xi$   from $\{x_0, \cdot, x_{K-1}\}$ uniformly at random.
\end{algorithm}

In order to solve the composite optimization problems under the generalized Polyak-{\L}ojasiewicz condition, we propose a variant of Prox-SpiderBoost, which we refer to as Prox-SpiderBoost-PL, described in \Cref{alg:Proximal_SPIDER+_sc}. 
We note that Prox-SpiderBoost-PL  
   can also be viewed as a generalization of SARAH \cite{Lam2017b} to a proximal algorithm with further differences lying in a much larger stepsize than that chosen by SARAH and random sampling with replacement for inner loop iterations, as opposed to sampling without replacement taken by SARAH.

Next, we present the convergence rate characterization of \Cref{alg:Proximal_SPIDER+_sc} for solving composite optimization problems under the generalized Polyak-{\L}ojasiewicz condition. 
\begin{theorem}\label{convex_thm}
	Let Assumprion \ref{assum: f} hold and apply Prox-SpiderBoost-PL in \Cref{alg:Proximal_SPIDER+_sc} to solve the problem (Q) with $\mathcal{X} = \mathbb{R}^d$.
	Assume the objective function satisfies the Polyak-{\L}ojasiewicz condition with parameter $\tau$ and set $q=|S|=\Theta(L\tau), \eta   = \frac{1}{8L}$. Then, the generated variable sequence satisfies, for all $t=1,2,...$
	\begin{align*}
	\mathbb{E} \|G_{\eta}(x_{tq})\|^2 \le \frac{64\tau L}{q-2} \mathbb{E} \|G_{\eta}(x_{(t-1)q})\|^2.
	\end{align*}
	Consequently, the oracle complexity of \Cref{alg:Proximal_SPIDER+_sc} for finding a point $x_\xi$ that satisfies $\mathbb{E} \|G_{\eta}(x_{\xi})\|  \le \epsilon$ is in the order $\mathcal{O}((n+L^2\tau^2)\log \frac{1}{\epsilon})$.
\end{theorem}
\Cref{convex_thm} shows that Prox-SpiderBoost-PL in \Cref{alg:Proximal_SPIDER+_sc} converges linearly to a stationary point for solving composite optimization problems under the generalized Polyak-{\L}ojasiewicz condition. 
We compare the SFO complexity in \Cref{convex_thm} with those of previous proposed stochastic proximal algorithms in \Cref{comparison_nonsmooth_gd}. Our result outperforms the state-of-art result in the regime of $\tau < n^{2/3}$, which is desirable for solving large data problem (i.e., $n$ is   large). Moreover, we note that both the results of ProxSVRG and ProxSVRG+ requires the condition number to satisfy $L\tau \geqslant   \sqrt{n}$ , 
 whereas our result of Prox-SpiderBoost-PL does not require the aforementioned condition, and has the most relaxed dependency on $n, L$ and $\tau$  demonstrating the superior performance of Prox-SpiderBoost-PL for optimizing functions under Polyak-{\L}ojasiewicz geometry.

\renewcommand{\arraystretch}{1.3} 
\definecolor{LightCyan}{rgb}{0.88,1,1}
\begin{table*}[h] 
	\small
	\centering 
	\caption{Comparison of results on SFO compelxity and PO compelxity  under Polyak-{\L}ojasiewicz condition.} \label{comparison_nonsmooth_gd}
	\vspace{2mm}
	\scalebox{0.9}{ 
		\begin{threeparttable} 
			\begin{tabular}{cllllc} \toprule
				\multirow{2}{*}{Algorithms}& &\multirow{2}{*}{Stepsize $\eta$} &\multicolumn{2}{c}{Finite-Sum} &\multirow{1}{*}{Additional}  \\ 
				\cmidrule{4-5}   
				&  & &\multicolumn{1}{c}{SFO} &\multicolumn{1}{c}{PO}  &\multirow{1}{*}{Condition} \\   \midrule
				
				ProxGD  &\citep{Karimi2016} &$\mathcal{O}(L^{-1})$  & $\mathcal{O}(n \tau \log(1/\epsilon))$   &  $\mathcal{O}(\tau \log(1/\epsilon))$ &- \\ \midrule
				
				ProxSVRG/SAGA &\citep{Reddi2016} &$\mathcal{O}(L^{-1})$  & $\mathcal{O}((n +  n^{2/3}\tau )\log(1/\epsilon))$   & $\mathcal{O}(\tau \log(1/\epsilon))$ & $L\tau \geqslant   \sqrt{n}$   \\ \midrule

				ProxSVRG$^\textbf{+}$ &\citep{Li2018}   &$\mathcal{O}(L^{-1})$ & $\mathcal{O}(  n^{2/3} \tau  \log(1/\epsilon))$& $\mathcal{O}(\tau \log(1/\epsilon))$ &$L\tau \geqslant   \sqrt{n}$  \\  \midrule
				
				\belowrulesepcolor{LightCyan}   
				\rowcolor{LightCyan}
				Prox-SpiderBoost-PL &(This Work) &$\mathcal{O}(L^{-1})$ & $\mathcal{O}((n +  \tau^2 )\log(1/\epsilon))$&  $\mathcal{O}(\tau \log(1/\epsilon))$ &- \\   \aboverulesepcolor{LightCyan}  \bottomrule
			\end{tabular}  
		\end{threeparttable} 
	}
\end{table*}  

For the case with $h=0$ (i.e., the problem objective reduces to the smooth function $f$), our algorithm achieves a total SFO complexity of $(n+L^2\tau^2)\log(1/\epsilon)$, which is the same as that achieved by SARAH \citep{Lam2017b}. However, we note that our algorithm allows to use a constant stepsize at the order of $\mathcal{O}(  1/{L} )$, whereas SARAH used a much smaller stepsize at the order of $\mathcal{O}(1/(L\sqrt{q}) )$.


\section{Prox-SpiderBoost-O  for Online Nonconvex Composite Optimization} \label{sec_online}
In this section, we study the performance of a variant of Prox-SpiderBoost for solving nonconvex composite optimization problems under the online setting.

\subsection{Unconstrained Optimization under Euclidean Geometry}
In this subsection, we study the following composite optimization problem.  
\begin{align}
\min_{x\in \mathcal{X}}	\Psi(x)  \defeq f(x)  + h(x),  f(x)=\mathbb{E}_{\zeta} [f_{\zeta}(x)]. \tag{R}
\end{align} 
Here the objective function $\Psi(x)$ consists of a population risk $\mathbb{E}_{\zeta} [f_{\zeta}(x)]$ over the underlying data distribution, a nonsmoooth but simple convex regularizer $h(x)$, and a convex constrain set $\mathcal{X}$. Such a problem can be viewed to have infinite samples as opposed to finite samples in the finite-sum problem (as in the problem (Q)), and the underlying data distribution is typically unknown a priori. Therefore, one cannot evaluate the full-gradient $\nabla f$ over the underlying data distribution in practice. For such a type of problems, we propose a variant of Prox-SpiderBoost, which applies stochastic sampling to estimate the full gradient for initializing the gradient estimator in each inner loop. We refer to this variant   as Prox-SpiderBoost-O, the details of which are summarized in \Cref{alg:Proximal_SPIDER+_online}. 
\begin{algorithm}
	\caption{Prox-SpiderBoost-O  for online optimization}
	\label{alg:Proximal_SPIDER+_online}  
	
	{\bf Input:} $\eta = \frac{1}{2L}$, $q, K, |S_1|, |S|\in \mathbb{N}$.
	
	{\bf \text{For} ~$k=0, 1, \ldots, K-1$}
	
	\quad {\textbf{If}~ $\textrm{mod}(k,q) = 0$:} 
	
	\qquad Draw $|S_1|$ samples with replacement.
	
	\qquad Set $v_k = \frac{1}{|S_1|} \sum_{i \in S_1} f_i(x_k)$.
	
	\quad {\bf{\text{Else}:} }
	
	\qquad Draw $|S|$ samples with replacement
	
	\qquad Compute $v_k$ according to \cref{spider}.
	
	\quad $x_{k+1} = \mathrm{prox}_{\eta h}(x_k - \eta v_k)$.
	
	\textbf{Output:} $x_\xi$   from $\{x_0, \cdot, x_{K-1}\}$ uniformly at random.
\end{algorithm}
It can be seen that Prox-SpiderBoost-O in \Cref{alg:Proximal_SPIDER+_online} draws $|S_1|$ stochastic samples to estimate the full gradient for initializing the gradient estimator. To analyze its performance, we introduce the following standard assumption on variance.
\begin{assum}\label{assum: variance}
The variance of stochastic gradients is bounded, i.e., there exists a constant $\sigma>0$ such that for all $x\in \mathbb{R}^d$ and all random draws of $\zeta$, it holds that $\mathbb{E}_\zeta \|\nabla f_\zeta(x) - \nabla f(x) \|^2 \leqslant \sigma^2$.
\end{assum}
Under \Cref{assum: variance}, the  variance of a mini-batch gradient with size $|S_1|$  can be  bounded by $\mathcal{O} ({\sigma^2}/{|S_1|})$. 
We obtain the following result on the oracle complexity for Prox-SpiderBoost-O in \Cref{alg:Proximal_SPIDER+_online}.
\begin{theorem}\label{online_thm}
	Let Assumptions \ref{assum: f} and \ref{assum: variance} hold and consider the problem (R) with $\mathcal{X} = \mathbb{R}^d$. Apply Prox-SpiderBoost-O with parameters $|S_1| = 24\sigma^2{\epsilon^{-2 }} , q =|S| = \sqrt{|S_1|}, \eta = \frac{1}{2L}$. Then, the corresponding output $x_\xi$ satisfies $\mathbb{E}\|G_{\eta}(x_\xi)\|  \le \epsilon$ provided that the total number $K$  of iterations satisfies 
	\begin{align*}
	K \ge \mathcal{O}\Big(\frac{L(\Psi(x_0) -  \Psi^* )}{\epsilon^2}\Big).
	\end{align*}
	Moreover, the resulting total SFO complexity is $\mathcal{O}( \epsilon^{-3 } + \epsilon^{-2})$, and the PO complexity is $\mathcal{O}(\epsilon^{-2})$.
\end{theorem}

To the best of our knowledge, the SFO complexity of \Cref{alg:Proximal_SPIDER+_online} improves the state-of-art result  $\mathcal{O} (\epsilon^{-10/3})$ \citep{Li2018,Allen_Zhu2017}  of online stochastic composite optimization by a factor of $\epsilon^{1/3}$. 

 In the smooth case with $h(x) = 0$, the problem (R) reduces to the online case of problem (P), and \Cref{alg:Proximal_SPIDER+_online} reduces to a online version of SpiderBoost. We refer to such an algorithm as SpiderBoost-O. The following corollary characterizes the performance of SpiderBoost-O to solve an online problem. 
\begin{coro} \label{coro_1}
Let Assumptions \ref{assum: f} and \ref{assum: variance} hold and consider the online setting of problem (P). Apply  SpiderBoot-O with parameters $|S_1| = 24\sigma^2{\epsilon^{-2 }} , q =|S| = \sqrt{|S_1|}, \eta = \frac{1}{2L}$. Then, the corresponding output $x_\xi$ satisfies $\mathbb{E}\| \nabla f(x_\xi)\| \le \epsilon$ provided that the total number $K$  of  iterations satisfies 
	\begin{align*}
	K \ge \mathcal{O}\Big(\frac{L(\Psi(x_0) -  \Psi^* )}{\epsilon^2}\Big).
	\end{align*}
	Moreover, the resulting total SFO complexity is $\mathcal{O}( \epsilon^{-3 } + \epsilon^{-2})$, and the PO complexity is $\mathcal{O}(\epsilon^{-2})$.
\end{coro}

\subsection{Constrained Optimization under Non-Euclidean Geometry} \label{nonsmooth_rls_o}
 \Cref{alg:Proximal_SPIDER+_online} can be generalized to solve the online optimization problem (R) subject to a convex constraint set $\mathcal{X}$ with a general distance function. To do this, one replaces the proximal gradient update in \Cref{alg:Proximal_SPIDER+_online} with the generalized proximal gradient step in \cref{p_d_1} which is based on a proper Bregman distance $V$. For such an algorithm, we obtain the following result on the oracle complexity for Prox-SpiderBoost-O in solving constrained stochastic composite optimization under non-Euclidean geometry.
\begin{theorem} \label{p_online_thm}
	Let Assumptions \ref{assum: f} and \ref{assum: variance} hold and consider the problem (R). Apply Prox-SpiderBoost-O with a proper Bregman distance $V$ that is $\alpha$-strongly convex with $\alpha>\frac{7}{8}$. Choose the parameters as $|S_1| =   2\left( \left(\alpha - \frac{7}{8}\right)^{-1}\left(1 + \frac{1}{4\alpha^2}\right) + \frac{2}{\alpha^2} \right)\sigma^2{\epsilon^{-2 }}$, $\eta = \frac{1}{2L}$, and $q=|S|=\sqrt{|S_1|} $. Then, the algorithm outputs a point $x_\xi$ that satisfies $\mathbb{E}\|G_{\eta}(x_{\xi})\|  \le \epsilon$ provided that the total number $K$  of iterations satisfies
	\begin{align*}
	K \ge \frac{8L( \Psi(x_0) -  \Psi^*)}{\epsilon^2 } \left({\alpha - \frac{7}{8}} \right)^{-1}\left(1 + \frac{1}{4\alpha}\right) .
	\end{align*}
	Moreover, the overall SFO complexity is $\mathcal{O}( \epsilon^{-3 }  + \epsilon^{-2})$ and the PO complexity is $\mathcal{O}(\epsilon^{-2})$.
\end{theorem}

\section{Prox-SpiderBoost-M-O for Online Nonconvex Composite Optimization}

As the online problem (R) depends on the population risk that contains infinite samples, we propose a variant of Prox-SpiderBoost-M that can solve it in an online setting. We summarize the detailed steps of the algorithm in \Cref{alg: online-ProxSPIDERM}, where we refer to it as  Prox-SpiderBoost-M-O.
\begin{algorithm}
	\caption{Prox-SpiderBoost-M-O}
\label{alg: online-ProxSPIDERM}
	{\bf Input:} $q, K \in \mathbb{N}, \{\lambda_k\}_{k=1}^{K-1}, \{\beta_k\}_{k=1}^{K-1} >0.$
	
	{\bf Set:} $\alpha_{k} = \frac{2}{\ceil[]{k/q}+1}$.
	
	{\bf Initialize:} $y_0 = x_0\in \mathbb{R}^d$.
	
	\For{$k=0, 1, \ldots, K-1$}
	{
		$z_{k} = (1-\alpha_{k+1})y_{k} + \alpha_{k+1} x_{k}$, \\
		\eIf{$\text{mod}(k, q)= 0$}
		{
			draw $\xi_1$ data samples  and compute  
			$v_{k} = \frac{1}{|\xi_1|} \sum_{i=1}^{|\xi_1|} \nabla f_{u_i}(x)$
		}
		{
			draw $\xi_2$ data samples   and compute 
			$v_k = \frac{1}{|\xi_2|}\sum_{i=1}^{|\xi_2|} (\nabla f_{u_i} (z_k) - \nabla f_{u_i} (z_{k-1})) + v_{k-1}$.
		}
		$x_{k+1} =  \mathrm{prox}_{\lambda_{k} g}\big(x_{k} - \lambda_{k} v_k\big)$, \\
		$y_{k+1} = z_{k} - \frac{\beta_{k}}{\lambda_k}x_k + \frac{\beta_{k}}{\lambda_k}  \mathrm{prox}_{\lambda_{k} g}\big(x_{k} - \lambda_{k} v_k\big)$.
	}
	{\textbf{Output:} $z_\zeta$, where $\zeta \overset{\text{Unif}}{\sim}\{0,\ldots,K-1\}$.}
\end{algorithm}

Note that unlike the Prox-SpiderBoost-M for the finite-sum case, the Prox-SpiderBoost-M-O keeps drawing new data samples from the underlying distribution  to construct the gradient estimate $v_k$. To study its convergence guarantee, we make the following standard assumption on the variance of the random sampling. 
We next present the  convergence guarantee for Prox-SpiderBoost-M-O.
\begin{theorem}\label{thm: ProxSpiderO}
	Let Assumptions \ref{assum: f} and \ref{assum: variance} hold. Apply Prox-SpiderBoost-M-O (see \Cref{alg: online-ProxSPIDERM}) to solve the problem (R). Choose any desired accuracy $\epsilon>0$ and set parameters $\alpha_{k} = \frac{2}{k+1}, q=|\xi_2| = \sqrt{|\xi_1|} = \sqrt{\frac{2\sigma^2}{\epsilon^2}}$, $\beta_k \equiv \frac{1}{8L}$ and $\lambda_k \in [\beta_k, (1+\alpha_k)\beta_{k}]$. Then, the output $z_{\zeta}$ of the algorithm satisfies $\mathbb{E}\|G_{\lambda_{\zeta}}(z_{\zeta}, \nabla f(z_{\zeta}))\| \le \epsilon$ provided that the total number of iterations $K$ satisfies
	\begin{align}
	K \ge \Theta\bigg(\frac{L(F(x_0)-F^*)}{\epsilon^2} \bigg).
	\end{align}
	Moreover, the total number of stochastic gradient calls is at most $\Theta(\epsilon^{-3})$ and the total number of proximal mapping calls is at most $\Theta(\epsilon^{-2})$.
\end{theorem}
The orders of the results in \Cref{thm: ProxSpiderO} match those of state-of-arts \cite{Fang2018}. Our result demonstrates that the momentum scheme can be applied to facilitate the convergence of Prox-SpiderBoost for solving online nonsmooth and nonconvex problems with a provable convergence guarantee.

\section{Objective Functions in Experiments}\label{app:exp}

We specify the two objective functions that we adopt in our experiments. The nonsmooth problems are the regularized versions of these problems. The first problem is the logistic regression problem with a nonconvex regularizer, which takes the following form
\begin{align*}
\min_{w\in \mathbb{R}^d} f(w):=\frac{1}{n} \sum_{i=1}^{n} \ell(w^\intercal x_i, y_i)+ \alpha \sum_{i=1}^{d} \frac{ w_i^2}{1  + w_i^2 },
\end{align*}
where $x_i\in \mathbb{R}^d$ denotes the features and $y_i\in \{\pm 1 \}$ corresponds to the labels, and $\alpha = 0.1$. We set the loss $\ell$ to be the cross-entropy loss given by
$$\ell(w^\intercal x_i, y_i)=-y_i  \log \left(\frac{1}{1 + e^{-w^Tx_i}}\right).$$


The second loss function is the following nonconvex robust linear regression problem
\begin{align*}
\min_{w \in \mathbb{R}^d} f(w):=\frac{1}{n} \sum_{i=1}^{n}  \ell(y_i - w^\intercal x_i), 
\end{align*}
where we use the nonconvex loss function $\ell (x) :=\log (\frac{x^2}{2} + 1)$. 
\\

\noindent {\Large \textbf{Technical Proofs}} 
\section{Analysis of SpiderBoost (Proof of \Cref{finite_sum_thm})} \label{proof_of_spider_plus}
 
Throughout the paper, let $n_k = \lceil k/q \rceil$ such that $(n_k-1)q \le k \le n_k q - 1$. Next, we establish our main result that yields \Cref{finite_sum_thm}. 
\begin{theorem} \label{general_thm}
	Under \Cref{assum: f}, if the parameters $\eta, q$ and $S$   are chosen such that 
	\begin{align}
	\beta_1 \triangleq \frac{\eta}{2} - \frac{L\eta^2}{2}- \frac{\eta^3L^2q }{2|S| } >0,
	\end{align}
	and  if it holds that for $\textrm{mod}(k,q) =0$, we always have\begin{align}
	\mathbb{E}\|v_{k} - \nabla f(x_{k})\|^2 \leq \epsilon_1^2, \label{parameter_gerneral_setting_1}
	\end{align}
	then  the output point $x_\xi$ of SpiderBoost satisfies 
	\begin{align}
	\mathbb{E} \|\nabla f(x_\xi)\|^2  \le \frac{2}{K\beta_1} \left(1+   \frac{ L^2 \eta^2q}{ |S| }  \right)\left( f(x_0) -  f^*  \right) + \left( \frac{\eta}{ \beta_1} + 2 + \frac{ L^2 \eta^3q}{  |S|\beta_1} \right)\epsilon_1^2.
	\end{align}
\end{theorem}

\subsection{Proof of \Cref{general_thm}}
We first present an auxiliary lemma from \citep{Fang2018}.
\begin{lemma}[\citep{Fang2018}, Lemma 1]\label{SpiderBoost_lemma_Zhang}
	Let \Cref{assum: f} hold. The gradient estimator $v_k$  generated by \cref{spider} satisfies for all $(n_k-1)q +1 \le k \le n_kq-1$,
	\begin{align}
	\mathbb{E} \|v_k - \nabla f(x_k)\|^2 \le \frac{L^2}{|S|} \mathbb{E}\|x_k - x_{k-1}\|^2 + \mathbb{E}\|v_{k-1} - \nabla f(x_{k-1})\|^2.
	\end{align} 
\end{lemma}
Telescoping \Cref{SpiderBoost_lemma_Zhang} over $k$ from $(n_k-1)q +1$ to $k$, where $k\le n_kq-1$, we obtain that
\begin{align}
\mathbb{E} \|v_k - \nabla f(x_k)\|^2 &\leq \sum_{i=(n_k-1)q}^{k-1}\frac{L^2}{|S|} \mathbb{E}\|x_{i+1} - x_{i}\|^2 + \mathbb{E}\|v_{(n_k-1)q} - \nabla f(x_{(n_k-1)q})\|^2 \nonumber \\
&\leq  \sum_{i=(n_k-1)q}^{k}\frac{L^2}{|S|} \mathbb{E}\|x_{i+1} - x_{i}\|^2 + \mathbb{E}\|v_{(n_k-1)q} - \nabla f(x_{(n_k-1)q})\|^2.  \label{SpiderBoost_eq: variance bound}
\end{align}
We   note that the above inequality also holds for $k = (n_k-1)q$, which can be simply checked by plugging $k = (n_k-1)q$ into above inequality.

\begin{proof}
	By \Cref{assum: f}, the entire objective function $f$ is $L$-smooth, which further implies that 
	\begin{align}
	f(x_{k+1}) &\le f(x_k) + \inner{\nabla f(x_k)}{x_{k+1}-x_k} + \frac{L}{2}\|x_{k+1}-x_k\|^2 \nonumber\\
	&\overset{(i)}{=} f(x_k) - \eta\inner{\nabla f(x_k)}{v_k} + \frac{L\eta^2}{2}\|v_k\|^2 \nonumber\\
	&= f(x_k) - \eta\inner{\nabla f(x_k)-v_k}{v_k} -\eta \|v_k\|^2+ \frac{L\eta^2}{2}\|v_k\|^2 \nonumber\\
	&\overset{(ii)}{\le} f(x_k) + \frac{\eta}{2} \|\nabla f(x_k)-v_k\|^2  - (\frac{\eta}{2} - \frac{L\eta^2}{2}) \|v_k\|^2 \nonumber,
	\end{align}
	where (i) follows from the update rule of SpiderBoost, (ii) uses the inequality that $\inner{x}{y}\le ({\|x\|^2+\|y\|^2})/{2}$ for all $x,y\in \mathbb{R}^d$. Taking expectation on both sides of the above inequality yields that
	 
	\begin{align}
	\mathbb{E} &f(x_{k+1}) \nonumber \\
	 &\le \mathbb{E} f(x_k) + \frac{\eta}{2} \mathbb{E}\|\nabla f(x_k)-v_k\|^2  - (\frac{\eta}{2} - \frac{L\eta^2}{2}) \mathbb{E}\|v_k\|^2  \nonumber\\
	&\overset{(i)}{\le}  \mathbb{E} f(x_k) + \frac{\eta}{2} \sum_{i=(n_k-1)q}^{k }\frac{L^2}{|S|} \mathbb{E}\|x_{i+1} - x_{i}\|^2 + \frac{\eta}{2}\mathbb{E}\|v_{(n_k-1)q} - \nabla f(x_{(n_k-1)q})\|^2  - (\frac{\eta}{2} - \frac{L\eta^2}{2}) \mathbb{E}\|v_k\|^2  \nonumber\\
	&\overset{(ii)}{=} \mathbb{E} f(x_k) + \frac{\eta^3}{2} \sum_{i=(n_k-1)q}^{k}\frac{L^2}{|S|} \mathbb{E}\|v_i\|^2 + \frac{ \eta}{2} \epsilon_1^2-  (\frac{\eta}{2} - \frac{L\eta^2}{2}) \mathbb{E}\|v_k\|^2, \label{eq: 1}
	\end{align}
	where (i) follows from \cref{SpiderBoost_eq: variance bound}, and (ii) follows from \cref{parameter_gerneral_setting_1}, and the fact that $x_{k+1} = x_k - \eta v_k$. Next, telescoping \cref{eq: 1} over $k$ from $(n_k-1)q$ to $k$ where $k \le n_kq-1$ and noting that for $(n_k-1)q \le j \le  n_kq-1$, $n_j = n_k$ , we obtain
	\begin{align}
	\mathbb{E} &f(x_{k+1}) \nonumber \\
	 &\le \mathbb{E} f(x_{(n_k-1) q}) + \frac{\eta^3}{2} \sum_{j=(n_k - 1)q}^{k} \sum_{i=(n_k-1)q}^{j}\frac{L^2}{|S|} \mathbb{E}\|v_i\|^2 +\frac{\eta}{2}\sum_{j=(n_k-1)q}^{k} \epsilon_1^2 - (\frac{\eta}{2} - \frac{L\eta^2}{2}) \sum_{j=(n_k - 1)q}^{k}\mathbb{E}\|v_j\|^2 \nonumber \\ 
	&\overset{(i)}{\le}   \mathbb{E} f(x_{(n_k-1) q}) + \frac{\eta^3}{2} \sum_{j=(n_k - 1)q}^{k} \sum_{i=(n_k-1)q}^{k}\frac{L^2}{|S|} \mathbb{E}\|v_i\|^2 +\frac{\eta}{2}\sum_{j=(n_k-1)q}^{k} \epsilon_1^2 - (\frac{\eta}{2} - \frac{L\eta^2}{2}) \sum_{j=(n_k - 1)q}^{k}\mathbb{E}\|v_j\|^2 \nonumber \\
	&\overset{(ii)}{\le} \mathbb{E} f(x_{(n_k-1) q}) + \frac{\eta^3L^2q}{2|S|} \sum_{i=(n_k - 1)q}^{k}   \mathbb{E}\|v_i\|^2 +\frac{\eta  }{2}\sum_{j=(n_k-1)q}^{k} \epsilon_1^2- (\frac{\eta}{2} - \frac{L\eta^2}{2}) \sum_{j=(n_k - 1)q}^{k}\mathbb{E}\|v_j\|^2 \nonumber \\
	&= \mathbb{E} f(x_{(n_k-1) q}) - \sum_{i=(n_k - 1)q}^{k} \left(\frac{\eta}{2} - \frac{L\eta^2}{2}- \frac{\eta^3L^2q }{2|S| }\right) \mathbb{E}\|v_i\|^2 +\frac{\eta  }{2}\sum_{i=(n_k-1)q}^{k} \epsilon_1^2 \nonumber \\
	&\overset{(iii)}{=} \mathbb{E} f(x_{(n_k-1) q}) -  \sum_{i=(n_k - 1)q}^{k}  \left(\beta_1 \mathbb{E}\|v_i\|^2 - \frac{\eta  }{2}\epsilon_1^2 \right)\label{eq: 2}
	\end{align}
	where (i)    extends the summation of the second term from $j$ to $k$, (ii) follows from the fact that $k  \leqslant n_kq-1$. Thus, we obtain
	\begin{align*}
	\sum_{j=(n_k - 1)q}^{k} &\sum_{i=(n_k-1)q}^{k}\frac{L^2}{|S|} \mathbb{E}\|v_i\|^2 \le \frac{(k+q-n_kq +1 )L^2}{|S|} \sum_{i=(n_k-1)q}^{k}  \mathbb{E}\|v_i\|^2   \le \frac{qL^2}{|S|} \sum_{i=(n_k-1)q}^{k}  \mathbb{E}\|v_i\|^2,  
	\end{align*}
	and (iii) follows from the definition of $\beta_1$.
	
	We continue the proof by further driving
	\begin{align*}
	\mathbb{E} f(x_{K}) - &\mathbb{E}f(x_0) \\
	&= (\mathbb{E} f(x_{q}) - \mathbb{E}f(x_0)) + (\mathbb{E} f(x_{2q}) - \mathbb{E}f(x_q)) +\cdots + (	\mathbb{E} f(x_{K}) - \mathbb{E} f(x_{(n_k-1) q}))\\
	&\overset{(i)}{\leq} -   \sum_{i=0}^{q-1} \left(\beta_1 \mathbb{E}\|v_i\|^2 - \frac{\eta  }{2}\epsilon_1^2 \right) -    \sum_{i=q}^{2q-1}\left(\beta_1 \mathbb{E}\|v_i\|^2 - \frac{\eta  }{2}\epsilon_1^2 \right)- \cdots -   \sum_{i=(n_K -1)q}^{K-1} \left(\beta_1 \mathbb{E}\|v_i\|^2 - \frac{\eta  }{2}\epsilon_1^2 \right)\\
	&= -   \sum_{i=0}^{K-1 }  \left(\beta_1 \mathbb{E}\|v_i\|^2 - \frac{\eta  }{2}\epsilon_1^2 \right) = -   \sum_{i=0}^{K-1} \beta_1 \mathbb{E}\|v_i\|^2 + \frac{ K \eta}{2}\epsilon_1^2 ,
	\end{align*}
	where (i) follows from \cref{eq: 2}.
	Note that $\mathbb{E} f(x_{K}) \ge f^* \triangleq \inf_{x\in \mathbb{R}^d} f(x)$. Hence,  the above inequality implies that
	\begin{align}
	\sum_{i=0}^{K-1}  \beta_1 \mathbb{E}\|v_i\|^2    \le  f(x_0) -  f^* +  \frac{K\eta  }{2}\epsilon_1^2. \label{total_sum}
	\end{align}
	We next bound $\mathbb{E}  \|\nabla f(x_\xi)\|^2$, where $\xi$ is selected uniformly at random from $\{0,\ldots,K-1\}$. Observe that
	\begin{align}
	\mathbb{E}  \|\nabla f(x_\xi)\|^2  = \mathbb{E}  \|\nabla f(x_\xi) - v_\xi + v_\xi\|^2 \leq 2\mathbb{E}  \|\nabla f(x_\xi) - v_\xi \|^2 + 2\mathbb{E}  \| v_\xi \|^2 \label{final_0}.
	\end{align}
	Next, we bound the two terms on the right hand side of the above inequality. First, note that
	\begin{align}
	\mathbb{E} \|v_\xi\|^2  = \frac{1}{K }\sum_{i=0}^{K-1} \mathbb{E} \|v_i\|^2 \le  \frac{f(x_0) -  f^*}{K\beta_1 }  +  \frac{\eta }{2\beta_1}\epsilon_1^2, \label{eq: 5}
	\end{align}
	where the last inequality follows from \cref{total_sum}. On the other hand, note that
	\begin{align}
	\mathbb{E}  \|\nabla f(x_\xi)-v_\xi\|^2  &\overset{(i)}{\le}\mathbb{E}    \sum_{i=(n_\xi-1)q}^{\xi }\frac{L^2}{|S| } \mathbb{E}\|x_{i+1} - x_{i}\|^2 + \epsilon_1^2 \overset{(ii)}{=} \epsilon_1^2 +\mathbb{E}   \sum_{i=(n_\xi-1)q}^{\xi }\frac{L^2\eta^2}{ |S|   } \mathbb{E}\|v_i\|^2  \nonumber \\
	&\overset{(iii)}{\leq} \epsilon_1^2 + \mathbb{E}   \sum_{i=(n_\xi-1)q}^{\min\{(n_\xi)q   - 1, K-1\} }\frac{L^2\eta^2}{ |S| } \mathbb{E}\|v_i\|^2  \overset{(iv)}{\leq}  \epsilon_1^2 + \frac{q}{K } \sum_{i=0}^{K-1} \frac{L^2\eta^2}{|S|}\mathbb{E}\|v_{i}\|^2 \nonumber \\
	&\overset{(v)}{\le}  \epsilon_1^2 +  \frac{ L^2 \eta^2q}{K|S|\beta_1} \left( f(x_0) - f^*  \right) +  \frac{ L^2 \eta^3q}{2 |S|\beta_1}\epsilon_1^2, \label{final_2}
	\end{align}
	where (i) follows from \cref{SpiderBoost_eq: variance bound,parameter_gerneral_setting_1}, (ii) follows from the fact that $x_{k+1} = x_k - \eta v_k$, (iii) follows from the definition of $n_\xi$, which implies $\xi \leqslant \min \{(n_\xi) q-1,K-1  \} $, (iv) follows from the fact that the probability that $n_\xi = 1,2,\cdots,n_K $   is less than or equal to $q/(K)$, and (v) follows from \cref{total_sum}. 
	
	Substituting \cref{eq: 5,final_2} into \cref{final_0}, we obtain
	\begin{align*}
	\mathbb{E} \|\nabla f(x_\xi)\|^2  &\le  \frac{2\left(f(x_0) -  f^*\right)}{K\beta_1 }  +  \frac{\eta }{ \beta_1}\epsilon_1^2 +  2\epsilon_1^2 +  \frac{2 L^2 \eta^2q}{K|S|\beta_1} \left( f(x_0) - f^*  \right) +  \frac{ L^2 \eta^3q}{  |S|\beta_1}\epsilon_1^2\\
	&= \frac{2}{K\beta_1} \left(1+   \frac{ L^2 \eta^2q}{ |S| }  \right)\left( f(x_0) -  f^*  \right) + \left( \frac{\eta}{ \beta_1} + 2 + \frac{ L^2 \eta^3q}{  |S|\beta_1} \right)\epsilon_1^2. 
	\end{align*}
\end{proof}

\subsection{Proof of  \Cref{finite_sum_thm} }
Based on the parameter setting in \Cref{finite_sum_thm} that 
\begin{align}
  q =\sqrt{n} , S  = \sqrt{n}, \text{ and }  \eta = \frac{1}{2L}, \label{finite_sum_parameter_setting}
\end{align} 
we obtain
\begin{align}
\beta_1 = \frac{\eta}{2} - \frac{L\eta^2}{2}- \frac{\eta^3L^2q }{2|S| } = \frac{1}{16L} > 0. \label{finte_sum_cond_1}
\end{align}
Moreover, for $\textrm{mod}(k,q) =0$, as the algorithm is designed to take the full-batch gradient of the finite-sum problem, we have
\begin{align}
\mathbb{E} \|v_k - \nabla f(x_k)\|^2 = \mathbb{E} \| \nabla f(x_k) - \nabla f(x_k)\|^2 = 0. \label{finte_sum_cond_2}
\end{align}
\Cref{finte_sum_cond_1,finte_sum_cond_2} imply that the parameters in \Cref{finite_sum_thm}  satisfy the assumptions in \Cref{general_thm} with $\beta_1 = 1/(16L)$ and $\epsilon_1 = 0$. Plugging \cref{finite_sum_parameter_setting,finte_sum_cond_1,finte_sum_cond_2}   into \Cref{general_thm}, we obtain that, after $K$ iterations, the output of SpiderBoost satisfies
\begin{align}
\mathbb{E} \|\nabla f(x_\xi)\|^2  \le \frac{40L}{K }  \left( f(x_0) -  f^*  \right) .
\end{align} 
To ensure $\mathbb{E} \|\nabla f(x_\xi)\|   \leqslant \epsilon$, it is sufficient to ensure $\mathbb{E} \|\nabla f(x_\xi)\|^2   \leqslant \epsilon^2$ (because $\left( \mathbb{E} \|\nabla f(x_\xi)\|\right)^2\leq \mathbb{E} \|\nabla f(x_\xi)\|^2 $ due to Jensen's inequality). Thus, we need the total number $K$ of iterations  satisfies that  $ \frac{40L}{K }  \left( f(x_0) -  f^*  \right) \leq \epsilon^2$, which gives
\begin{align}
K =  \frac{40L}{\epsilon^2 }  \left( f(x_0) -  f^*  \right). \label{finite_sum_number_of_iteration}
\end{align}
Then, the total SFO complexity is given by
\begin{align*}
\left\lceil  \frac{K}{q} \right\rceil \cdot  n + K\cdot S \leqslant (K+q)\cdot \frac{n}{q} + K\cdot S =  K\sqrt{n} + n + K\sqrt{n}  = \mathcal{O}(\sqrt{n} \epsilon^{-2} + n),
\end{align*}
where the last equation follows from \cref{finite_sum_number_of_iteration}, which completes the proof.

\section{Analysis of Prox-SpiderBoost and Prox-SpiderBoost-O (Proofs of \Cref{p_finite_sum_thm,p_finite_sum_thm_nonEuclidean,online_thm,p_online_thm})}
We first establish the following major theorem, which is applicable to both the finite-sum and the online problem. We then generalize it for these two cases.
\begin{theorem} \label{p_main_theorem}
	Under \Cref{assum: f}, choose  a proper prox-funtion $V(\cdot): \mathcal{X} \rightarrow \mathbb{R}$ with modulus $\alpha$. Then, if the parameters $\eta, q$ and $S$   are chosen such that 
	\begin{align}
	\beta_2 \triangleq \alpha \eta  - \frac{L\eta^2}{2}-\frac{\eta}{2}- \frac{\eta^3L^2q }{2|S| } >0,
	\end{align}
	and  if it holds that for mod$(k,q) = 0$,  we always have
	\begin{align}
	\mathbb{E}\|v_{k} - \nabla f(x_{k})\|^2 \leq \epsilon_1^2,  \label{p_general_outer_bound}
	\end{align}
	then, the output point $x_\xi$ of Prox-SpiderBoost or Prox-SpiderBoost-O satisfies 
	\begin{align}
	\mathbb{E} \|\tilde{g}_\xi\|^2 \le  \frac{2}{K\beta_2} \left(1+   \frac{ L^2 \eta^2q}{\alpha^2 |S| }  \right)\left( f(x_0) -  f^*  \right) + \left( \frac{\eta}{ \beta_2} + \frac{2}{\alpha^2} + \frac{ L^2 \eta^3q}{ \alpha^2 |S|\beta_2} \right)\epsilon_1^2,
	\end{align}
	where $\tilde{g}_\xi =  P_{\mathcal{X}}(x_\xi,\nabla f(x_\xi), \eta) $.
\end{theorem} 
As stated in the theorem, we require  $\beta = \left(\alpha \eta - \frac{L\eta^2}{2} - \frac{\eta}{2} - \frac{\eta^3 L^2 }{2 }    \right) >0$ to conclude our theorem. A simple case would be $\eta =  {1}/{(2L)}$ and $w(x) = \|x\|^2/2 $, which gives  $\alpha = 1$ and $\beta = 1/(16L)$. 
 
\subsection{Proof of \Cref{p_main_theorem}}
To prove \Cref{p_main_theorem}, we first introduce a useful lemma.
\begin{lemma}[\citep{Ghadimi2016}, Lemma 1 and Proposition 1]
	Let $\mathcal{X}$ be a closed convex set in $\mathbb{R}^d$, $h:\mathcal{X} \rightarrow \mathbb{R}$ be a convex function, but possibly nonsmooth,  and $V: \mathcal{X} \rightarrow \mathbb{R}$ be defined in \cref{p_strongly_convex}. Moreover,  define
	\begin{align}
	x^+ &= \arg \min_{u\in \mathcal{X}} \left\{\inner{g}{u} + \frac{1}{\eta}V(u,x) + h(u) \right\} \\
	P_{\mathcal{X}}(x,g,\eta) &= \frac{1}{\eta}(x - x^+), \label{p_gradient}
	\end{align}
	where $g  \in \mathbb{R}^d$, $x \in \mathcal{X}$, and  $\eta > 0$. Then, the following statement  hold 
	\begin{align}
	\inner{g}{P_{\mathcal{X}}(x,g,\eta)} &\geqslant \alpha \|P_{\mathcal{X}}(x,g,\eta)\|^2 + \frac{1}{\eta}[h(x^+) - h(x)]. \label{p_subProblem_inequality}
	\end{align} 
	Moreover, for any $g_1,g_2 \in \mathbb{R}^d$, we have
	\begin{align}
	\|P_{\mathcal{X}}(x,g_1,\eta) - P_{\mathcal{X}}(x,g_2,\eta)   \| &\leqslant \frac{1}{\alpha} \|g_1 - g_2\|. \label{p_lipschitz_g}
	\end{align}
\end{lemma}

Now, we are ready to prove \Cref{p_main_theorem}. To ease our nation, let $g_k=P_{\mathcal{X}}(x_k,v_k,\eta)$, which is defined in \cref{p_gradient}.
	We begin with the analysis at iteration $k$. By the Lipschitz continuity of $\nabla f$, we obtain
	\begin{align}
	f(x_{k+1}) &\le f(x_k) + \inner{\nabla f(x_k)}{x_{k+1}-x_k} + \frac{L}{2}\|x_{k+1}-x_k\|^2 \nonumber\\
	&\overset{(i)}{=} f(x_k) - \eta\inner{\nabla f(x_k)}{g_k} + \frac{L\eta^2}{2}\|g_k\|^2 \nonumber\\
	&= f(x_k) - \eta\inner{\nabla f(x_k)-v_k}{g_k} -\eta \inner{v_k}{g_k} + \frac{L\eta^2}{2}\|g_k\|^2 \nonumber\\
	&\overset{(ii)}{\le} f(x_k) + \frac{\eta}{2} \|\nabla f(x_k)-v_k\|^2     -\eta \inner{v_k}{g_k} + \left(\frac{L\eta^2}{2} + \frac{\eta}{2} \right) \|g_k\|^2\nonumber \\
	&\overset{(iii)}{\le} f(x_k) + \frac{\eta}{2} \|\nabla f(x_k)-v_k\|^2     -\alpha \eta \|g_k\|^2 +h(x_k)-h(x_{k+1}) +  \left(\frac{L\eta^2}{2} + \frac{\eta}{2} \right) \|g_k\|^2 \label{p_step_1}
	\end{align}
	where (i) follows from the definition of $P_{\mathcal{X}}(x_k,v_k,\eta)$, and the update rule of Prox-SpiderBoost and Prox-SpiderBoost-O, (ii) follows from the inequality that $\inner{x}{y}\le \frac{\|x\|^2+\|y\|^2}{2}$ for $x,y\in \mathbb{R}^d$, and (iii) follows from \cref{p_subProblem_inequality} with $g = v_k, x = x_k$ and $P_{\mathcal{X}}(x_k,v_k,\eta) = g_k$.

Taking expectation on both sides of  \cref{p_step_1}, and arranging it with the definition of  $	\Psi(x)  \defeq f(x)  + h(x)$, we obtain
\begin{align*}
\mathbb{E} \Psi(x_{k+1}) &\leq  \mathbb{E} \Psi(x_{k}) +\frac{\eta}{2}\mathbb{E}  \|\nabla f(x_k)-v_k\|^2    - \left(\alpha \eta - \frac{L\eta^2}{2} - \frac{\eta}{2} \right) \mathbb{E} \|g_k\|^2\\
&\overset{(i)}{\leq} \mathbb{E} \Psi(x_{k}) +\frac{\eta }{2}\sum_{i=(n_k-1)q}^{k} \frac{L^2}{|S|}\mathbb{E}\|x_{i+1} - x_{i}\|^2 \\
&\qquad \qquad +  \frac{\eta}{2}\mathbb{E}\|v_{(n_k-1)q} - \nabla f(x_{(n_k-1)q})\|^2   - \left(\alpha \eta - \frac{L\eta^2}{2} - \frac{\eta}{2} \right) \mathbb{E} \|g_k\|^2 \\
&\overset{(ii)}{\leq} \mathbb{E} \Psi(x_{k}) +\frac{\eta^3  }{2 }\sum_{i=(n_k-1)q}^{k} \frac{L^2}{|S|}\mathbb{E}\|g_i\|^2 + \frac{\eta \epsilon_1^2}{2}  - \left(\alpha \eta - \frac{L\eta^2}{2} - \frac{\eta}{2} \right) \mathbb{E} \|g_k\|^2
\end{align*}
where (i) follows from \cref{SpiderBoost_eq: variance bound}, and (ii) follows from \cref{p_general_outer_bound} and the fact that $x_{k+1} = x_k - \eta g_k$. Telescoping the above inequality over $k$ from $(n_k-1)q$ to $k$ where $k \le n_kq-1$ and noting that for $(n_k-1)q \le j \le  n_kq-1$, $n_j = n_k$, we have
\begin{align}
\mathbb{E}  \Psi(x_{k+1}) - & \mathbb{E} \Psi(x_{(n_k-1) q}) \nonumber  \\
&\le  \frac{\eta^3}{2} \sum_{j=(n_k - 1)q}^{k} \sum_{i=(n_k-1)q}^{j}\frac{L^2}{|S|} \mathbb{E}\|g_i\|^2 +\frac{\eta}{2}\sum_{j=(n_k-1)q}^{k} \epsilon_1^2 - \left(\alpha \eta - \frac{L\eta^2}{2} - \frac{\eta}{2} \right) \sum_{j=(n_k - 1)q}^{k}\mathbb{E}\|g_j\|^2 \nonumber \\ 
&\overset{(i)}{\le}  \frac{\eta^3}{2} \sum_{j=(n_k - 1)q}^{k} \sum_{i=(n_k-1)q}^{k}\frac{L^2}{|S|} \mathbb{E}\|g_i\|^2 +\frac{\eta}{2}\sum_{j=(n_k-1)q}^{k} \epsilon_1^2 - \left(\alpha \eta - \frac{L\eta^2}{2} - \frac{\eta}{2} \right) \sum_{j=(n_k - 1)q}^{k}\mathbb{E}\|g_j\|^2 \nonumber \\
&\overset{(ii)}{\le}   \frac{\eta^3L^2q}{2|S|} \sum_{i=(n_k - 1)q}^{k}   \mathbb{E}\|g_i\|^2 +\frac{\eta  }{2}\sum_{j=(n_k-1)q}^{k} \epsilon_1^2- \left(\alpha \eta - \frac{L\eta^2}{2} - \frac{\eta}{2} \right) \sum_{j=(n_k - 1)q}^{k}\mathbb{E}\|g_j\|^2 \nonumber \\
&= - \sum_{i=(n_k - 1)q}^{k} \left(\alpha \eta  - \frac{L\eta^2}{2}-\frac{\eta}{2}- \frac{\eta^3L^2q }{2|S| }\right) \mathbb{E}\|g_i\|^2 +\frac{\eta  }{2}\sum_{i=(n_k-1)q}^{k} \epsilon_1^2 \nonumber \\
&\overset{(iii)}{=}   -  \sum_{i=(n_k - 1)q}^{k}  \left(\beta_2 \mathbb{E}\|g_i\|^2 - \frac{\eta  }{2}\epsilon_1^2 \right) \label{p_per_iteration}
\end{align}
where (i)    extends the summation of second term from $j$ to $k$, (ii) follows from the fact that $k  \leqslant n_kq-1$ and thus
\begin{align*}
\sum_{j=(n_k - 1)q}^{k} &\sum_{i=(n_k-1)q}^{k}\frac{L^2}{|S|} \mathbb{E}\|g_i\|^2 \le \frac{(k+q-n_kq +1 )L^2}{|S|} \sum_{i=(n_k-1)q}^{k}  \mathbb{E}\|g_i\|^2   \le \frac{qL^2}{|S|} \sum_{i=(n_k-1)q}^{k}  \mathbb{E}\|g_i\|^2,  
\end{align*}
and (iii) follows from the definition of $\beta_2$. We continue to derive
\begin{align*}
\mathbb{E} \Psi(x_{K}) - &\mathbb{E}\Psi(x_0) \\
 &= (\mathbb{E} \Psi(x_{q}) - \mathbb{E}\Psi(x_0)) + (\mathbb{E} \Psi(x_{2q}) - \mathbb{E}\Psi(x_q)) +\cdots + (	\mathbb{E} \Psi(x_{K}) - \mathbb{E} \Psi(x_{(n_k-1) q}))\\
&\overset{(i)}{\leq} -   \sum_{i=0}^{q-1} \left(\beta_2 \mathbb{E}\|g_i\|^2 - \frac{\eta  }{2}\epsilon_1^2 \right) -    \sum_{i=q}^{2q-1}\left(\beta_2 \mathbb{E}\|g_i\|^2 - \frac{\eta  }{2}\epsilon_1^2 \right)- \cdots -   \sum_{i=(n_K -1)q}^{K-1} \left(\beta_2 \mathbb{E}\|g_i\|^2 - \frac{\eta  }{2}\epsilon_1^2 \right)\\
&= -   \sum_{i=0}^{K-1 }  \left(\beta_2 \mathbb{E}\|g_i\|^2 - \frac{\eta  }{2}\epsilon_1^2 \right) = -   \sum_{i=0}^{K-1} \beta_2 \mathbb{E}\|g_i\|^2 + \frac{ K \eta}{2}\epsilon_1^2 ,
\end{align*}
where (i) follows from \cref{p_per_iteration}.
Note that $\mathbb{E} \Psi(x_{K}) \ge \Psi^* \triangleq \inf_{x\in \mathbb{R}^d} \Psi(x)$. The   above inequality implies that
\begin{align}
\sum_{i=0}^{K-1}  \beta_2 \mathbb{E}\|g_i\|^2    \le  \Psi(x_0) -  \Psi^* +  \frac{K\eta  }{2}\epsilon_1^2. \label{p_total_sum}
\end{align}

We next bound the output of algorithms. Define $\tilde{g}_\xi =   P(x_\xi,\nabla f(x_\xi), \eta) $, where $\xi$ is selected uniformly at random from $\{0,\ldots,K-1\}$. Observe that
\begin{align}
\mathbb{E} \|\tilde{g}_\xi\|^2 &\leq 2\mathbb{E} \| g_\xi \|^2 + 2\mathbb{E} \|\tilde{g}_\xi - g_\xi\|^2 \overset{(i)}{\leq} 2\mathbb{E} \| g_\xi \|^2 + \frac{2}{\alpha^2} \mathbb{E} \|\nabla f(x_\xi) - v_\xi\|^2  \label{p_final_0}
\end{align}
where (i) follows from the definition of $\tilde{g}_k,g_k$ and the property of $g_k$ and $\tilde{g}_k$ in \cref{p_lipschitz_g}.
 
Next, we bound the two terms on the right hand side of the above inequality. First, note that
\begin{align}
\mathbb{E} \|g_\xi\|^2  = \frac{1}{K }\sum_{i=0}^{K-1} \mathbb{E} \|g_i\|^2 \le  \frac{\Psi(x_0) -  \Psi^*}{K\beta_2 }  +  \frac{\eta }{2\beta_2}\epsilon_1^2, \label{p_eq: 5}
\end{align}
where the last inequality follows from \cref{p_total_sum}. On the other hand, note that
\begin{align}
\mathbb{E}  \|\nabla f(x_\xi)-v_\xi\|^2  &\overset{(i)}{\le}\mathbb{E}    \sum_{i=(n_\xi-1)q}^{\xi }\frac{L^2}{|S| } \mathbb{E}\|x_{i+1} - x_{i}\|^2 + \epsilon_1^2 \overset{(ii)}{=} \epsilon_1^2 +\mathbb{E}   \sum_{i=(n_\xi-1)q}^{\xi }\frac{L^2\eta^2}{ |S|   } \mathbb{E}\|g_i\|^2  \nonumber \\
&\overset{(iii)}{\leq} \epsilon_1^2 + \mathbb{E}   \sum_{i=(n_\xi-1)q}^{\min\{(n_\xi)q   - 1, K-1\} }\frac{L^2\eta^2}{ |S| } \mathbb{E}\|g_i\|^2  \overset{(iv)}{\leq}  \epsilon_1^2 + \frac{q}{K } \sum_{i=0}^{K-1} \frac{L^2\eta^2}{|S|}\mathbb{E}\|g_{i}\|^2 \nonumber \\
&\overset{(v)}{\le}  \epsilon_1^2 +  \frac{ L^2 \eta^2q}{K|S|\beta_2} \left( \Psi(x_0) - \Psi^*  \right) +  \frac{ L^2 \eta^3q}{2 |S|\beta_2}\epsilon_1^2, \label{p_final_2}
\end{align}
where (i) follows from \cref{SpiderBoost_eq: variance bound,p_general_outer_bound}, (ii) follows from the fact that $x_{k+1} = x_k - \eta g_k$, (iii) follows from the definition of $n_\xi$, which implies $\xi \leqslant \min \{(n_\xi) q-1,K-1  \} $, (iv) follows from the fact that the probability that $n_\xi = 1$ or $2$ ,  $\cdots$ or $n_k $   is less than or equal to $q/K$, and (v) follows from \cref{p_eq: 5}.

Substituting \cref{p_eq: 5,p_final_2} into \cref{p_final_0} yields
\begin{align*}
\mathbb{E} \|\tilde{g}_\xi \|^2  &\le  \frac{2\left(f(x_0) -  f^*\right)}{K\beta_2 }  +  \frac{\eta }{ \beta_2}\epsilon_1^2 +  \frac{2}{\alpha^2}\left( \epsilon_1^2 +  \frac{  L^2 \eta^2q}{K|S|\beta_2} \left( f(x_0) - f^*  \right) +  \frac{ L^2 \eta^3q}{  2|S|\beta_2}\epsilon_1^2\right)\\
&= \frac{2}{K\beta_2} \left(1+   \frac{ L^2 \eta^2q}{\alpha^2 |S| }  \right)\left( \Psi(x_0) -  \Psi^*  \right) + \left( \frac{\eta}{ \beta_2} + \frac{2}{\alpha^2} + \frac{ L^2 \eta^3q}{ \alpha^2 |S|\beta_2} \right)\epsilon_1^2 ,
\end{align*}
which completes the proof.

\subsection{Proof of \Cref{p_finite_sum_thm}}
 \begin{proof}
 	\Cref{p_finite_sum_thm} as a special case follows from the more general \Cref{p_finite_sum_thm_nonEuclidean} that we develop in \Cref{general_smooth_rlt} with the choices of the Bregman distance function $V(x,y) = \frac{1}{2} \|x - y\|^2$ and $\alpha = 1$.  
 \end{proof}

\subsection{Proof of \Cref{p_finite_sum_thm_nonEuclidean}}
Based on the parameter setting in \Cref{p_finite_sum_thm} that 
\begin{align}
\alpha > \frac{7}{8},  q =\sqrt{n} , S = \sqrt{n}, \text{ and }  \eta = \frac{1}{2L}, \label{p_finite_sum_parameter_setting}
\end{align} 
we obtain
\begin{align}
\beta_1 =  \alpha \eta  - \frac{L\eta^2}{2}-\frac{\eta}{2}- \frac{\eta^3L^2q }{2|S| } =  \frac{1}{2L}(\alpha - \frac{7}{8}) > 0. \label{p_finte_sum_cond_1}
\end{align}
Moreover, for mod$(k,q) =0$, as the algorithm is designed to take the full-batch gradient of the finite-sum problem, we have
\begin{align}
\mathbb{E} \|v_k - \nabla f(x_k)\|^2 = \mathbb{E} \| \nabla f(x_k) - \nabla f(x_k)\|^2 = 0. \label{p_finte_sum_cond_2}
\end{align}
\Cref{p_finte_sum_cond_1,p_finte_sum_cond_2} imply that the parameters   in  the finite-sum case satisfy the assumptions in \Cref{p_main_theorem} with $\beta_1 =  (\alpha -  {7}/{8})/(2L)$ and $\epsilon_1 = 0$. Plugging \cref{p_finite_sum_parameter_setting,p_finte_sum_cond_1,p_finte_sum_cond_2}   into \Cref{p_main_theorem}, we obtain that, after $K$ iterations, the output of Prox-SpiderBoost satisfies
\begin{align}
\mathbb{E} \| \tilde{g}_\xi\|^2  \le \frac{4L}{K } \left({\alpha - \frac{7}{8}} \right)^{-1}\left(1 + \frac{1}{4\alpha}\right) \left( \Psi(x_0) -  \Psi^*  \right) .
\end{align} 
To ensure $\mathbb{E} \|\tilde{g}_\xi\|  \leqslant \epsilon$, it is sufficient to ensure  $\mathbb{E} \|\tilde{g}_\xi\|^2  \leqslant \epsilon^2$, thus, we obtain  \begin{align}
K =  \frac{4L}{\epsilon^2 } \left({\alpha - \frac{7}{8}} \right)^{-1}\left(1 + \frac{1}{4\alpha}\right) \left( \Psi(x_0) -  \Psi^*  \right). \label{p_finite_sum_number_of_iteration}
\end{align}
Then, the SFO is 
\begin{align*}
\left\lceil  \frac{K}{q} \right\rceil \cdot  n + K\cdot S \leqslant (K+q)\cdot \frac{n}{q} + K\cdot S =  K\sqrt{n} + n + K\sqrt{n}  = \mathcal{O}(\sqrt{n} \epsilon^{-2} + n),
\end{align*}
where the last equation follows from \cref{finite_sum_number_of_iteration}. The proximal oracle follows from the total iteration in \cref{p_finite_sum_number_of_iteration},  which completes the proof.

\subsection{Proof of \Cref{online_thm}}
\begin{proof}
	\Cref{online_thm} follows as a special case from the more general \Cref{p_online_thm}	that we develop in \cref{nonsmooth_rls_o} with the choices of Bregman distance function $V(x,y) = \frac{1}{2} \|x - y\|^2$ and $\alpha = 1$.
\end{proof}

\subsection{Proof of \Cref{p_online_thm}}
Based on the parameter setting in \Cref{p_online_thm} that 
\begin{align}
\alpha > \frac{7}{8} , S_1 =   2\left( \left(\alpha - \frac{7}{8}\right)^{-1}\left(1 + \frac{1}{4\alpha^2}\right) + \frac{2}{\alpha^2} \right)\sigma^2{\epsilon^{-2 }} , q =\sqrt{S_1} , S = \sqrt{S_1}, \text{ and }  \eta = \frac{1}{2L},  \label{p_online_parameter_setting}
\end{align} 
we obtain
\begin{align}
\beta_2 =  \alpha \eta  - \frac{L\eta^2}{2}-\frac{\eta}{2}- \frac{\eta^3L^2q }{2|S| } =  \frac{1}{2L}(\alpha - \frac{7}{8})   > 0. \label{p_online_cond_1}
\end{align}
Moreover, for mod$(k,q) =0$, we have
\begin{align}
\mathbb{E} \|v_k - \nabla f(x_k)\|^2 &= \mathbb{E} \left\| \frac{1}{|S_1|}\sum_{i \in S_1} \nabla f_i(x_k) - \nabla f(x_k) \right\|^2 =  \frac{1}{| S_1|^2} \left\| \sum_{i \in S_1} \nabla f_i(x_k) - \nabla f(x_k)\right\|^2 \\
&\overset{(i)}{=}   \frac{1}{| S_1|^2} \sum_{i \in S_1} \left\|   \nabla f_i(x_k) - \nabla f(x_k)\right\|^2 =  \frac{1}{| S_1| } \left\|   \nabla f_i(x_k) - \nabla f(x_k)\right\|^2 \overset{(ii)}{=} \frac{\sigma^2}{|S_1|} \\
&\overset{(iii)}{\le}   \left( \left(\alpha - \frac{7}{8}\right)^{-1}\left(1 + \frac{1}{4\alpha^2}\right) + \frac{2}{\alpha^2} \right)^{-1} {\frac{\epsilon^2 }{2}} . \label{p_online_cond_2}
\end{align}
where (i) follows from $\mathbb{E} \nabla f_i(x_k) -  \nabla f(x_k) = 0 $, and the fact that the samples from $S_1$ are drawn with replacement, and (iii) follows from \cref{p_online_parameter_setting}.

 \Cref{p_online_cond_1,p_online_cond_2} imply that the parameters  in  the online case satisfy the assumptions in \Cref{p_main_theorem} with $\beta_2 =   (\alpha -  {7}/{8})/(2L) $ and $\epsilon_1^2 = \left( \left(\alpha - \frac{7}{8}\right)^{-1}\left(1 + \frac{1}{4\alpha^2}\right) + \frac{2}{\alpha^2} \right)^{-1} {\frac{\epsilon^2 }{2}}$. Plugging \cref{p_online_parameter_setting,p_online_cond_1,p_online_cond_2}   into \Cref{p_main_theorem}, we obtain that, after $K$ iterations, the output of Prox-SpiderBoost-O satisfies
\begin{align}
\mathbb{E} \|\tilde{g}_\xi\|^2 &\le  \frac{2}{K\beta_2} \left(1+   \frac{ L^2 \eta^2q}{\alpha^2 |S| }  \right)\left( \Psi(x_0) -  \Psi^*  \right) + \left( \frac{\eta}{ \beta_2} + \frac{2}{\alpha^2} + \frac{ L^2 \eta^3q}{ \alpha^2 |S|\beta_2} \right)\epsilon_1^2\\
&= \frac{4L}{K } \left({\alpha - \frac{7}{8}} \right)^{-1}\left(1 + \frac{1}{4\alpha}\right) \left( \Psi(x_0) -  \Psi^*  \right)  +  \frac{\epsilon^2}{2}.
\end{align} 
To ensure $\mathbb{E} \| \tilde{g}_\xi\|   \leqslant \epsilon$, it is sufficient to ensure $\mathbb{E} \| \tilde{g}_\xi\|^2   \leqslant \epsilon^2$, thus, we need 
\begin{align}
K =  \frac{8L}{\epsilon^2 } \left({\alpha - \frac{7}{8}} \right)^{-1}\left(1 + \frac{1}{4\alpha}\right) \left( \Psi(x_0) -  \Psi^*  \right) . \label{p_online_number_of_iteration}
\end{align}
Then, the total SFO complexity is 
\begin{align*}
\left\lceil  \frac{K}{q} \right\rceil \cdot  S_1 + K\cdot S \leqslant (K+q)\cdot \frac{S_1}{q} + K\cdot S =  K\sqrt{S_1} + S_1+ K\sqrt{S_1}  = \mathcal{O}( \epsilon^{-3 } + \epsilon^{-2}  ),
\end{align*}
where the last equation follows from \cref{p_online_number_of_iteration}. The proximal oracle follows from the total iteration in \cref{p_online_number_of_iteration}, which finishes the proof.

\subsection{Proof of \Cref{coro_1} }
\begin{proof}
	\Cref{coro_1} follows directly from \Cref{online_thm}, becasue the online setting  of problem (P) is a special case of the problem (R).
\end{proof}

\section{Analysis of Prox-SpiderBoost-PL (Proof of \Cref{convex_thm})}
Let us consider one outer loop. Following a similar proof as that of eq.(25) in \citep{Li2018}, we obtain the following inequality for Prox-SpiderBoost-PL in finite-sum case.
\begin{align*}
\mathbb{E} \Psi(x_{k+1}) \le \mathbb{E} \big[\Psi(x_{k}) - (\frac{1}{2\eta} - \frac{L}{2})\|x_{k+1}-x_k\|^2 - (\frac{1}{3\eta}-L)\|\overline{x_{k+1}} - x_k\|^2 +\eta\|\nabla f(x_k)-v_k\|^2 \big],
\end{align*}
where $\overline{x_{k+1}} := \mathrm{prox}_{\eta g}(x_k - \eta \nabla f(x_k))$. Substituting the variance bound of Spider into the above inequality we obtain that
\begin{align*}
\mathbb{E} \Psi(x_{k+1}) \le \mathbb{E} \big[\Psi(x_{k}) - (\frac{1}{2\eta} - \frac{L}{2})\|x_{k+1}-x_k\|^2 - (\frac{1}{3\eta}-L)\|\overline{x_{k+1}} - x_k\|^2 +\eta\sum_{i=(n_k-1)q}^{k}\frac{L^2}{|S|} \|x_{i+1}-x_i\|^2 \big].
\end{align*}
Summing the above inequality over $k$ from $(n_k-1)q$ to $n_kq-2$ and relax the upper bound of $i$ to $n_kq-2$, we further obtain that
\begin{align*}
\mathbb{E} \Psi(x_{n_k q-1}) \le \mathbb{E} \Psi(x_{(n_k-1) q}) -  \sum_{i=(n_k-1)q}^{n_kq-2}(\frac{1}{2\eta} - \frac{L}{2} - \frac{\eta L^2(q-2)}{|S|})\mathbb{E}\|x_{i+1}-x_i\|^2 - (\frac{1}{3\eta}-L)\sum_{i=(n_k-1)q}^{n_kq-2}\mathbb{E}\|\overline{x_{i+1}} - x_i\|^2.
\end{align*}
Noting that $q=|S|, \eta L = \frac{1}{8}$, we further obtain that
\begin{align*}
\mathbb{E} \Psi(x_{n_k q-1}) &\le \mathbb{E} \Psi(x_{(n_k-1) q}) - \sum_{i=(n_k-1)q}^{n_kq-2}3L\mathbb{E}\|x_{i+1}-x_i\|^2 - L\eta^2\sum_{i=(n_k-1)q}^{n_kq-2}\mathbb{E}\|\mathcal{G}_\eta(x_i)\|^2 \\
&\le \mathbb{E} \Psi(x_{(n_k-1) q}) - L\eta^2\sum_{i=(n_k-1)q}^{n_kq-2}\mathbb{E}\|\mathcal{G}_\eta(x_i)\|^2 
\end{align*}
Since $\mathbb{E} \Psi(x_{n_k q-1})\ge \Psi^*$, the above inequality further implies that
\begin{align*}
\sum_{i=(n_k-1)q}^{n_kq-2}\mathbb{E}\|\mathcal{G}_\eta(x_i)\|^2  &\le 64L(\mathbb{E} \Psi(x_{(n_k-1) q}) - \Psi^*).
\end{align*}
By the scheme of Prox-SpiderBoost-gd, we know that $\mathbb{E}\|\mathcal{G}_\eta(x_{n_k q})\|^2 = \frac{1}{q-2}\sum_{i=(n_k-1) q+1}^{n_k q-2}\mathbb{E}\|\mathcal{G}_\eta(x_{i})\|^2$. Therefore, combining this inequality with the above inequality, we obtain that
\begin{align*}
\mathbb{E}\|\mathcal{G}_\eta(x_{n_k q})\|^2 &= \frac{1}{q-2}\sum_{i=(n_k-1) q+1}^{n_k q-2}\mathbb{E}\|\mathcal{G}_\eta(x_{i})\|^2 \\
&\le \frac{64L}{q-2}(\mathbb{E} \Psi(x_{(n_k-1) q}) - \Psi^*) \\
&\le \frac{64L\tau}{q-2} \mathbb{E} \|\mathcal{G}_\eta(x_{(n_k-1) q})\|^2.
\end{align*}

In order to produce a point such that $\mathbb{E}\|\mathcal{G}_\eta(x_{t q})\| \le \epsilon$, we deduce from the above inequality that at least $t= \Theta(\log \frac{1}{\epsilon} / \log \frac{q}{L\tau})$ number of outer loops is needed. Note that $|S| = q=\Theta(L\tau)$, we conclude that $t= \Theta(\log \frac{1}{\epsilon})$. In summary, the total proximal oracle complexity (PO) is in the order $O(q\log \frac{1}{\epsilon}) = O(\tau\log \frac{1}{\epsilon})$, and the total stochastic first-order oracle complexity (SFO) is $O((n+q|S|)\log \frac{1}{\epsilon}) = O((n+L^2\tau^2)\log \frac{1}{\epsilon})$.

\section{Analysis of Prox-SpiderBoost-M and SpiderBoost-M-O (Proof of \Cref{thm: ProxSpiderM} and  \Cref{thm: ProxSpiderO})}
\subsection{Proof of \Cref{thm: ProxSpiderM}}
In this section, we provide the convergence analysis of Prox-SpiderBoost-M. 
Throughout, for any $k\in \mathbb{N}$, denote $\tau(k) \in \mathbb{N}$ the unique integer such that $(\tau(k)-1)q \le k \le \tau(k) q - 1$. We also define $\Gamma_0 = 0, \Gamma_1 = 1$ and $\Gamma_k = (1-\alpha_k)\Gamma_{k-1}$ for $k=2,3,...$. Since we set $\alpha_k = \frac{2}{\ceil[]{k/q}+1}$, it is easy to check that $\Gamma_k = \frac{2}{\ceil[]{k/q}(\ceil[]{k/q}+1)}$.
We first provide some auxiliary lemmas that are useful for the analysis later.

\subsection*{Auxiliary Lemmas}

We first present an auxiliary lemma from \cite{Fang2018}.
\begin{lemma}{\cite{Fang2018}}\label{lemma: Zhang}
	Under \Cref{assum: f}, the estimation  $v_k$ of gradient constructed by SPIDER satisfies that for all $(\tau(k)-1)q +1 \le k \le \tau(k)q-1$,
	\begin{align*}
	\mathbb{E} \|v_k - \nabla f(z_k)\|^2 &\le \frac{L^2}{|\xi_k|} \mathbb{E}\|z_k - z_{k-1}\|^2 +\mathbb{E}\|v_{k-1} - \nabla f(z_{k-1})\|^2.
	\end{align*}
\end{lemma}

Telescoping \Cref{lemma: Zhang} and noting that $v_k = \nabla f(z_k)$ for all $k$ such that $\textrm{mod}(k,q) =0$, we obtain the following bound.
\begin{lemma}\label{aux: 1}
	Under \Cref{assum: f}, the estimation  $v_k$ of gradient constructed by SPIDER satisfies that for all $k\in \mathbb{N}$,
	\begin{align}
	\mathbb{E} \|v_k - \nabla f(z_k)\|^2 \le \sum_{i=(\tau(k)-1)q}^{k-1}\frac{L^2}{|\xi_i|} \mathbb{E}\|z_{i+1} - z_{i}\|^2. \label{eq: variance bound}
	\end{align}
\end{lemma}

Next, recall the following definition of the gradient mapping for some $\eta>0$ and $x, u \in \mathbb{R}^d$:
\begin{align*}
 G_\eta(x,u) := \frac{1}{\eta} \big(x-\mathrm{prox}_{\eta h}(x-\eta u) \big).
\end{align*}
Based on this definition, we can rewrite the updates of \Cref{alg: ProxSPIDERM} as follows:
\begin{align}
	z_{k} &= (1-\alpha_{k+1})y_{k} + \alpha_{k+1} x_{k}, \nonumber\\
	x_{k+1} &=  x_k - \lambda_{k} G_{\lambda_k}(x_k, v_k), \nonumber\\
    y_{k+1} &= z_{k} - \beta_{k} G_{\lambda_k}(x_k, v_k). \nonumber
\end{align}
Next, we prove the following auxiliary lemma. 
\begin{lemma}\label{aux: 3}
	Let the sequences $\{x_k\}_k, \{y_k\}_k, \{z_k\}_k$ be generated by \Cref{alg: ProxSPIDERM}. Then, the following inequalities hold
	\begin{align}
	y_k - x_k &= \Gamma_k\sum_{t=1}^k \frac{\lambda_{t-1} - \beta_{t-1}}{\Gamma_t} G_{\lambda_{t-1}}(x_{t-1}, v_{t-1}),\\
	\|y_k - x_k\|^2 &\le \Gamma_k\sum_{t=1}^k \frac{\lambda_{t-1} - \beta_{t-1}}{\alpha_t\Gamma_t} \|G_{\lambda_{t-1}}(x_{t-1}, v_{t-1})\|^2,\\
	\|z_{k+1}-z_k\|^2 &\le 2\beta_{k}^2\|G_{\lambda_{k}}(x_{k}, v_{k})\|^2 + 2\alpha_{k+2}^2 \Gamma_{k+1}\sum_{t=1}^{k+1} \frac{(\lambda_{t-1} - \beta_{t-1})^2}{\alpha_t\Gamma_t} \|G_{\lambda_{t-1}}(x_{t-1}, v_{t-1})\|^2.
	\end{align}
\end{lemma}
\begin{proof}
	We prove the first equality. By the update rule of the momentum scheme, we obtain that
	\begin{align}
	y_k - x_k &= z_{k-1} - \beta_{k-1} G_{\lambda_{k-1}}(x_{k-1}, v_{k-1}) - (x_{k-1} - \lambda_{k-1} G_{\lambda_{k-1}}(x_{k-1}, v_{k-1})) \nonumber\\
	&= (1-\alpha_k) (y_{k-1} - x_{k-1}) + (\lambda_{k-1} - \beta_{k-1}) G_{\lambda_{k-1}}(x_{k-1}, v_{k-1}). 
	\end{align}
	Dividing both sides by $\Gamma_k$ and noting that $\frac{1-\alpha_k}{\Gamma_k} = \frac{1}{\Gamma_{k-1}}$, we further obtain that
	\begin{align}
	\frac{y_k - x_k}{\Gamma_k} &= \frac{y_{k-1} - x_{k-1}}{\Gamma_{k-1}} + \frac{\lambda_{k-1} - \beta_{k-1}}{\Gamma_k} G_{\lambda_{k-1}}(x_{k-1}, v_{k-1}). 
	\end{align}
	Telescoping the above equality over $k$ yields the first desired equality. 
	
	Next, we prove the second inequality. Based on the first equality, we obtain that
	\begin{align}
	\|y_k - x_k \|^2 &= \|\Gamma_{k} \sum_{t=1}^{k} \frac{\lambda_{t-1} - \beta_{t-1}}{\Gamma_t} G_{\lambda_{t-1}}(x_{t-1}, v_{t-1})\|^2 \nonumber \\
	&= \|\Gamma_{k} \sum_{t=1}^{k} \frac{\alpha_t}{\Gamma_t} \frac{\lambda_{t-1} - \beta_{t-1}}{\alpha_t} G_{\lambda_{t-1}}(x_{t-1}, v_{t-1})\|^2 \nonumber \\
	&\overset{(i)}{\le} \Gamma_{k} \sum_{t=1}^{k} \frac{\alpha_t}{\Gamma_t} \frac{(\lambda_{t-1} - \beta_{t-1})^2}{\alpha_t^2}  \|G_{\lambda_{t-1}}(x_{t-1}, v_{t-1})\|^2 \nonumber \\
	&= \Gamma_{k} \sum_{t=1}^{k} \frac{(\lambda_{t-1} - \beta_{t-1})^2}{\Gamma_t\alpha_t}  \|G_{\lambda_{t-1}}(x_{t-1}, v_{t-1})\|^2,
	\end{align}
	where (i) uses the facts that $\{\Gamma_k\}_k$ is a decreasing sequence, $\sum_{t=1}^{k} \frac{\alpha_t}{\Gamma_t} = \frac{1}{\Gamma_k}$ and Jensen's inequality.
	
	Finally, we prove the third inequality. By the update rule of the momentum scheme, we obtain that $z_{k+1} - z_{k} = y_{k+1} - z_k + \alpha_{k+2} (x_{k+1} - y_{k+1})$. Then, we further obtain that
	\begin{align}
	\|z_{k+1} - z_{k}\| &\le \|y_{k+1} - z_k\| + \alpha_{k+2} \|x_{k+1} - y_{k+1}\| \nonumber \\
	&\le \beta_{k} \|G_{\lambda_k}(x_k, v_k)\| + \alpha_{k+2} \sqrt{\|x_{k+1} - y_{k+1}\|^2} \nonumber \\
	&\le \beta_{k} \|G_{\lambda_k}(x_k, v_k)\| + \alpha_{k+2} \sqrt{\Gamma_{k+1} \sum_{t=1}^{k+1} \frac{(\lambda_{t-1} - \beta_{t-1})^2}{\Gamma_t\alpha_t}  \|G_{\lambda_{t-1}}(x_{t-1}, v_{t-1})\|^2} \nonumber.
	\end{align}
	The desired result follows by taking the square on both sides of the above inequality and using the fact that $(a+b)^2 \le 2a^2+2b^2$.
\end{proof}

We also need the following lemma, which was established as Lemma 1 and Proposition 1 in \cite{Ghadimi2016}.
\begin{lemma}[Lemma 1 and Proposition 1, \cite{Ghadimi2016}]\label{aux: 4}
	Let $g$ be a proper and closed convex function. Then, for all $u, v, x\in \mathbb{R}^d$ and $\eta>0$, the following statements hold:
	\begin{align*}
		&\inner{u}{G_{\eta}(x,u)} \ge \|G_{\eta}(x,u)\|^2 + \frac{1}{\eta} \big(g(\mathrm{prox}_{\eta g}(x-\eta u)) - g(x) \big), \\
		&\|G_{\eta}(x,u) - G_{\eta}(x,v)\| \le \|u-v\|.
	\end{align*}
\end{lemma}

{\em Proof of \Cref{thm: ProxSpiderM}}: 

Consider any iteration $k$ of the algorithm. By smoothness of $f$, we obtain that
\begin{align}
f(x_k) &\le f(x_{k-1}) + \inner{\nabla f(x_{k-1})}{x_k - x_{k-1}} + \frac{L}{2}\|x_k - x_{k-1}\|^2 \nonumber \\
&= f(x_{k-1}) + \inner{\nabla f(x_{k-1})}{- \lambda_{k-1} G_{\lambda_{k-1}}(x_{k-1}, v_{k-1})} + \frac{L\lambda_{k-1}^2}{2}\|G_{\lambda_{k-1}}(x_{k-1}, v_{k-1})\|^2 \nonumber \\
&= f(x_{k-1}) - \lambda_{k-1} \inner{\nabla f(x_{k-1})-v_{k-1}}{G_{\lambda_{k-1}}(x_{k-1}, v_{k-1})} -  \lambda_{k-1} \inner{v_{k-1}}{G_{\lambda_{k-1}}(x_{k-1}, v_{k-1})} \nonumber\\
&\quad+ \frac{L\lambda_{k-1}^2}{2}\|G_{\lambda_{k-1}}(x_{k-1}, v_{k-1})\|^2 \nonumber\\
&\overset{(i)}{\le} f(x_{k-1}) - \lambda_{k-1} \inner{\nabla f(x_{k-1})-v_{k-1}}{G_{\lambda_{k-1}}(x_{k-1}, v_{k-1})} -  \lambda_{k-1} \|G_{\lambda_{k-1}}(x_{k-1}, v_{k-1})\|^2  \nonumber\\ 
&\quad - \big(h(\mathrm{prox}_{\lambda_{k-1} h}(x_{k-1}-\lambda_{k-1} v_{k-1})) - h(x_{k-1}) \big) + \frac{L\lambda_{k-1}^2}{2}\|G_{\lambda_{k-1}}(x_{k-1}, v_{k-1})\|^2 \nonumber\\
&= f(x_{k-1}) - \lambda_{k-1} \inner{\nabla f(x_{k-1})-v_{k-1}}{G_{\lambda_{k-1}}(x_{k-1}, v_{k-1})} -  \lambda_{k-1} \|G_{\lambda_{k-1}}(x_{k-1}, v_{k-1})\|^2  \nonumber\\ 
&\quad - \big(h(x_k) - h(x_{k-1}) \big) + \frac{L\lambda_{k-1}^2}{2}\|G_{\lambda_{k-1}}(x_{k-1}, v_{k-1})\|^2, \nonumber
\end{align}
where (i) follows from \Cref{aux: 4}. Rearranging the above inequality and using Cauchy-Swartz inequality yields that
\begin{align}
	\Psi(x_k) \le  \Psi(x_{k-1}) - \lambda_{k-1}(1-\frac{L\lambda_{k-1}}{2}) \|G_{\lambda_{k-1}}(x_{k-1}, v_{k-1})\|^2 + \lambda_{k-1} \|\nabla f(x_{k-1})-v_{k-1}\| \|G_{\lambda_{k-1}}(x_{k-1}, v_{k-1})\|. \label{eq: 15}
\end{align}

Note that
\begin{align}
\|\nabla f(x_{k-1}) - v_{k-1}\| &\le \|\nabla f(x_{k-1}) - \nabla f(z_{k-1})\| + \|\nabla f(z_{k-1}) - v_{k-1}\| \nonumber \\
&\overset{(i)}{\le} L\|x_{k-1} - z_{k-1}\| + \|\nabla f(z_{k-1}) - v_{k-1}\| \nonumber \\
&\overset{(ii)}{\le} L(1-\alpha_k) \|y_{k-1} - x_{k-1}\| + \|\nabla f(z_{k-1}) - v_{k-1}\| \nonumber,
\end{align}
where (i) uses the Lipschitz continuity of $\nabla f$ and (ii) follows from the update rule of the momentum scheme. Substituting the above inequality into \cref{eq: 15} yields that
\begin{align}
\Psi(x_k) &\le \Psi(x_{k-1}) - \lambda_{k-1}(1-\frac{L\lambda_{k-1}}{2})\|G_{\lambda_{k-1}}(x_{k-1}, v_{k-1})\|^2 + L\lambda_{k-1}(1-\alpha_k) \|G_{\lambda_{k-1}}(x_{k-1}, v_{k-1})\|\|y_{k-1} - x_{k-1}\| \nonumber\\
&\quad + \lambda_{k-1} \|G_{\lambda_{k-1}}(x_{k-1}, v_{k-1})\|\|\nabla f(z_{k-1}) - v_{k-1}\| \nonumber \\
&\le \Psi(x_{k-1}) - \lambda_{k-1}(1-\frac{L\lambda_{k-1}}{2})\|G_{\lambda_{k-1}}(x_{k-1}, v_{k-1})\|^2 + \frac{L\lambda_{k-1}^2}{2} \|G_{\lambda_{k-1}}(x_{k-1}, v_{k-1})\|^2 \nonumber\\ 
&\qquad+ \frac{L(1-\alpha_k)^2}{2}\|y_{k-1} - x_{k-1}\|^2 
+ \frac{\lambda_{k-1}}{2} \|G_{\lambda_{k-1}}(x_{k-1}, v_{k-1})\|^2 + \frac{\lambda_{k-1}}{2}\|\nabla f(z_{k-1}) - v_{k-1}\|^2 \nonumber \\
&= \Psi(x_{k-1}) - \lambda_{k-1}(\frac{1}{2}-L\lambda_{k-1}) \|G_{\lambda_{k-1}}(x_{k-1}, v_{k-1})\|^2 + \frac{L(1-\alpha_k)^2}{2}\|y_{k-1} - x_{k-1}\|^2 \nonumber\\
&\quad+ \frac{\lambda_{k-1}}{2}\|\nabla f(z_{k-1}) - v_{k-1}\|^2 \nonumber\\
&\le  \Psi(x_{k-1}) - \lambda_{k-1}(\frac{1}{2}-L\lambda_{k-1}) \|G_{\lambda_{k-1}}(x_{k-1}, v_{k-1})\|^2 + \frac{L\Gamma_{k-1}}{2}\sum_{t=1}^{k-1} \frac{\lambda_{t-1} - \beta_{t-1}}{\alpha_t\Gamma_t} \|G_{\lambda_{t-1}}(x_{t-1}, v_{t-1})\|^2 \nonumber\\
&\quad+ \frac{\lambda_{k-1}}{2}\|\nabla f(z_{k-1}) - v_{k-1}\|^2, \nonumber
\end{align}
where the last inequality uses item 2 of \Cref{aux: 3} and the fact that $0<\alpha_k <1$.
Telescoping the above inequality over $k$ from $1$ to $K$ yields that
\begin{align}
\Psi(x_{K}) &\le \Psi(x_{0}) - \sum_{k=0}^{K-1} \lambda_k (\frac{1}{2}-L\lambda_k) \|G_{\lambda_{k}}(x_{k}, v_{k})\|^2 + \sum_{k=0}^{K-1} \frac{L\Gamma_{k}}{2}\sum_{t=0}^{k-1} \frac{(\lambda_{t} - \beta_{t})^2}{\Gamma_{t+1}\alpha_{t+1}}  \|G_{\lambda_{t}}(x_{t}, v_{t})\|^2 \nonumber\\
&\quad+ \sum_{k=0}^{K-1} \frac{\lambda_k}{2}\|\nabla f(z_{k}) - v_k\|^2 \nonumber \\
&= \Psi(x_{0}) - \sum_{k=0}^{K-1} \lambda_k (\frac{1}{2}-L\lambda_k) \|G_{\lambda_{k}}(x_{k}, v_{k})\|^2 + \frac{L}{2} \sum_{k=0}^{K-1} \frac{(\lambda_k - \beta_k)^2}{\Gamma_{k+1}\alpha_{k+1}} \|G_{\lambda_{k}}(x_{k}, v_{k})\|^2(\sum_{t=k}^{K-1} \Gamma_{t})  \nonumber\\
&\quad+ \sum_{k=0}^{K-1} \frac{\lambda_k}{2}\|\nabla f(z_{k}) - v_k\|^2, \label{eq: 30}
\end{align}
where we have exchanged the order of summation in the second equality. Furthermore, note that $\sum_{t=k}^{K-1} \Gamma_{t} = 2\sum_{t=k}^{K-1} \frac{1}{\ceil[]{t/q}} - \frac{1}{\ceil[]{t/q}+1} \le \frac{2}{\ceil[]{k/q}}$. Then, substituting this bound into the above inequality and taking expectation on both sides yield that
\begin{align}
\mathbb{E}[\Psi(x_{K})] &\le \Psi(x_{0}) - \sum_{k=0}^{K-1} \lambda_k (\frac{1}{2}-L\lambda_k) \mathbb{E}\|G_{\lambda_{k}}(x_{k}, v_{k})\|^2 + \frac{L}{2} \sum_{k=0}^{K-1} \frac{2(\lambda_k - \beta_k)^2}{\ceil[]{k/q}\Gamma_{k+1}\alpha_{k+1}} \mathbb{E}\|G_{\lambda_{k}}(x_{k}, v_{k})\|^2  \nonumber\\
&\quad+ \sum_{k=0}^{K-1} \frac{\lambda_k}{2} \mathbb{E}\|\nabla f(z_{k}) - v_k\|^2. \label{eq: 16}
\end{align}

Next, we bound the term $\mathbb{E}\|\nabla f(z_{k}) - v_k\|^2$ in the above inequality. By \Cref{aux: 1} we obtain that
\begin{align}
\mathbb{E}\|\nabla f(z_{k}) - v_k\|^2 &\le \sum_{i=(\tau(k)-1)q}^{k-1}\frac{L^2}{|\xi_i|} \mathbb{E}\|z_{i+1} - z_{i}\|^2 \nonumber \\
&\le \sum_{i=(\tau(k)-1)q}^{k-1}\frac{L^2}{|\xi_i|} \big[2\beta_i^2\|G_{\lambda_{i}}(x_{i}, v_{i})\|^2 + 2\alpha_{i+2}^2 \Gamma_{i+1}\sum_{t=0}^{i} \frac{(\lambda_{t} - \beta_{t})^2}{\alpha_{t+1}\Gamma_{t+1}} \|G_{\lambda_{t}}(x_{t}, v_{t})\|^2 \big], \label{eq: 17}
\end{align}
where the last inequality uses item 3 of \Cref{aux: 3}. Substituting \cref{eq: 17} into \cref{eq: 16} and simplifying yield that
\begin{align}
\mathbb{E}[\Psi(x_{K})] &\le \Psi(x_{0}) - \sum_{k=0}^{K-1} \Big[\lambda_k (\frac{1}{2}-L\lambda_k) - \frac{L(\lambda_k - \beta_k)^2}{\ceil[]{k/q}\Gamma_{k+1}\alpha_{k+1}}\Big] \mathbb{E}\|G_{\lambda_{k}}(x_{k}, v_{k})\|^2 \nonumber\\
&\quad + \underbrace{\sum_{k=0}^{K-1} \frac{\lambda_k}{2} \mathbb{E}\bigg[\sum_{i=(\tau(k)-1)q}^{k-1}\frac{L^2}{|\xi_i|} \bigg[2\beta_i^2\|G_{\lambda_{i}}(x_{i}, v_{i})\|^2 + 2\alpha_{i+2}^2 \Gamma_{i+1}\sum_{t=0}^i \frac{(\lambda_t - \beta_t)^2}{\alpha_{t+1}\Gamma_{t+1}} \|G_{\lambda_{t}}(x_{t}, v_{t})\|^2 \bigg] \bigg]}_{T}. \label{eq: 18}
\end{align}
Before we proceed the proof, we first specify the choices of all the parameters. Specifically, we choose a constant mini-batch size $|\xi_k|\equiv |\xi|$, a constant $q=|\xi|$, a constant $\beta_k \equiv \beta>0$, $\lambda_k \in [\beta, (1+\alpha_{k+1})\beta]$. Based on these parameter settings, the term $T$ in the above inequality can be bounded as follows.
\begin{align}
T &\overset{(i)}{\le} \sum_{k=0}^{K-1} \frac{\lambda_k}{2} \mathbb{E}\bigg[\sum_{i=(\tau(k)-1)q}^{\tau(k)q-1}\frac{L^2}{|\xi_i|} \bigg[2\beta_i^2\|G_{\lambda_{i}}(x_{i}, v_{i})\|^2 + 2\alpha_{i+2}^2 \Gamma_{i+1}\sum_{t=0}^{k-1} \frac{(\lambda_t - \beta_t)^2}{\alpha_{t+1}\Gamma_{t+1}} \|G_{\lambda_{t}}(x_{t}, v_{t})\|^2 \bigg] \bigg] \nonumber\\
&\overset{(ii)}{\le} \sum_{k=0}^{K-1} \frac{\lambda_k L^2q\beta^2}{|\xi|}\mathbb{E}\|G_{\lambda_{k}}(x_{k}, v_{k})\|^2 + \sum_{k=0}^{K-1} \frac{2\lambda_kL^2}{|\xi|\tau(k)^3} \sum_{t=0}^{k-1}\frac{(\lambda_t - \beta_t)^2}{\alpha_{t+1}\Gamma_{t+1}}\mathbb{E}\|G_{\lambda_{t}}(x_{t}, v_{t})\|^2  \nonumber\\
&\overset{(iii)}{\le} \sum_{k=0}^{K-1} \lambda_k L^2\beta^2 \mathbb{E}\|G_{\lambda_{k}}(x_{k}, v_{k})\|^2 + \frac{2L^2\beta^2}{|\xi|} \sum_{k=0}^{K-1}  \frac{\alpha_{k+1}}{\Gamma_{k+1}} \mathbb{E}\|G_{\lambda_{k}}(x_{k}, v_{k})\|^2 (\sum_{t=k}^{K-1} \frac{\lambda_k}{\tau(t)^3})   \nonumber\\
&\overset{(iv)}{\le} \sum_{k=0}^{K-1} \lambda_k L^2\beta^2 \mathbb{E}\|G_{\lambda_{k}}(x_{k}, v_{k})\|^2 + \frac{4L^2\beta^3}{|\xi|} \sum_{k=0}^{K-1} (\ceil[]{k/q}+1) \mathbb{E}\|G_{\lambda_{k}}(x_{k}, v_{k})\|^2 (\sum_{t=(\tau(k)-1)q}^{\tau(K)q} \frac{1}{\tau(t)^3})   \nonumber\\
&= \sum_{k=0}^{K-1} \lambda_k L^2\beta^2 \mathbb{E}\|G_{\lambda_{k}}(x_{k}, v_{k})\|^2 + \frac{4L^2\beta^3}{|\xi|} \sum_{k=0}^{K-1} (\ceil[]{k/q}+1) \mathbb{E}\|G_{\lambda_{k}}(x_{k}, v_{k})\|^2 (\sum_{t=\tau(k)-1}^{\tau(K)} \frac{q}{(t+1)^3})  \nonumber\\
&\le \sum_{k=0}^{K-1} \lambda_k L^2\beta^2 \mathbb{E}\|G_{\lambda_{k}}(x_{k}, v_{k})\|^2 + 2L^2\beta^3 \sum_{k=0}^{K-1} (\ceil[]{k/q}+1) \mathbb{E}\|G_{\lambda_{k}}(x_{k}, v_{k})\|^2 \frac{1}{\tau(k)^2}   \nonumber \\
&\overset{(v)}{\le} \sum_{k=0}^{K-1} \lambda_k L^2\beta^2 \mathbb{E}\|G_{\lambda_{k}}(x_{k}, v_{k})\|^2 + 2L^2\beta^3 \sum_{k=0}^{K-1} \mathbb{E}\|G_{\lambda_{k}}(x_{k}, v_{k})\|^2 \frac{\ceil[]{k/q}+1}{\tau(k)^2} \nonumber\\
&\le \sum_{k=0}^{K-1} \lambda_k L^2\beta^2 \mathbb{E}\|G_{\lambda_{k}}(x_{k}, v_{k})\|^2 + 2L^2\beta^3 \sum_{k=0}^{K-1} \mathbb{E}\|G_{\lambda_{k}}(x_{k}, v_{k})\|^2, \label{eq: 19}
\end{align}
where (i) follows from the facts that $i\le k-1$ and $k-1\le \tau(k)q-1$, (ii) uses the fact that $\sum_{i=(\tau(k)-1)q}^{\tau(k)q-1} \alpha_{i+2}^2\Gamma_{i+1} \le \frac{2}{\tau(k)^3}$, (iii) uses the parameter settings $q=|\xi|$ and $\lambda_t-\beta_t \le \alpha_t\beta$, (iv) uses the facts that $\lambda_k\le 2\beta$ and $(\tau(k)-1)q \le k\le \tau(k)q$ and (v) uses the fact that $k \le \tau(k)q-1$. Substituting the above inequality into \cref{eq: 18} and simplifying, we obtain that
\begin{align}
\mathbb{E}[\Psi(x_{K})] &\le \Psi(x_{0}) - \sum_{k=0}^{K-1} \Big[\lambda_k (\frac{1}{2}-L\lambda_k - L^2\beta^2) - \frac{L(\lambda_k - \beta_k)^2}{\ceil[]{k/q}\Gamma_{k+1}\alpha_{k+1}} -2L^2\beta^3 \Big] \mathbb{E}\|G_{\lambda_{k}}(x_{k}, v_{k})\|^2 \label{eq: 31}\\
&\le \Psi(x_{0}) - \sum_{k=0}^{K-1} \Big[\beta (\frac{1}{2}-2L\beta - L^2\beta^2) - L\beta^2 -2L^2\beta^3 \Big] \mathbb{E}\|G_{\lambda_{k}}(x_{k}, v_{k})\|^2. \label{eq: 32}
\end{align}
Choosing $\beta \le \frac{1}{8L}$, the above inequality further implies that
\begin{align}
\mathbb{E} [\Psi(x_{K})] &\le \Psi(x_{0}) - \sum_{k=0}^{K-1} \frac{\beta}{16}\mathbb{E}\|G_{\lambda_{k}}(x_{k}, v_{k})\|^2 . \label{eq: 27}
\end{align}
Then, it follows that $\frac{1}{K}\sum_{k=0}^{K-1}\mathbb{E}\|G_{\lambda_{k}}(x_{k}, v_{k})\|^2 \le 16(\Psi(x_0) - \Psi^*)/(K\beta)$. Next, we bound the term $\mathbb{E}  \|G_{\lambda_{\zeta}}(z_{\zeta}, \nabla f(z_{\zeta}))\|^2$, where $\zeta$ is selected uniformly at random from $\{0,\ldots,K-1\}$. Observe that
\begin{align}
\mathbb{E}  \|G_{\lambda_{\zeta}}(z_{\zeta}, \nabla f(z_{\zeta}))\|^2  &= \mathbb{E}  \|G_{\lambda_{\zeta}}(z_{\zeta}, \nabla f(z_{\zeta})) -G_{\lambda_{\zeta}}(z_{\zeta}, v_{\zeta}) + G_{\lambda_{\zeta}}(z_{\zeta}, v_\zeta)\|^2 \nonumber\\
&\le 2\mathbb{E}  \|G_{\lambda_{\zeta}}(z_{\zeta}, \nabla f(z_{\zeta})) -G_{\lambda_{\zeta}}(z_{\zeta}, v_{\zeta})\|^2 + 2\mathbb{E}  \| G_{\lambda_{\zeta}}(z_{\zeta}, v_\zeta)\|^2 \nonumber\\
&\overset{(i)}{\le} 2\mathbb{E}  \|\nabla f(z_{\zeta}) - v_{\zeta}\|^2 + 2\mathbb{E}  \| G_{\lambda_{\zeta}}(z_{\zeta}, v_\zeta) - G_{\lambda_{\zeta}}(x_{\zeta}, v_\zeta) + G_{\lambda_{\zeta}}(x_{\zeta}, v_\zeta)\|^2 \nonumber\\
&\le 2\mathbb{E}  \|\nabla f(z_{\zeta}) - v_{\zeta}\|^2 + 4\mathbb{E}  \| G_{\lambda_{\zeta}}(z_{\zeta}, v_\zeta) - G_{\lambda_{\zeta}}(x_{\zeta}, v_\zeta)\|^2  + 4\mathbb{E}\|G_{\lambda_{\zeta}}(x_{\zeta}, v_\zeta)\|^2 \nonumber\\
&\le 2\mathbb{E}  \|\nabla f(z_{\zeta}) - v_{\zeta}\|^2 + 4\mathbb{E}\|G_{\lambda_{\zeta}}(x_{\zeta}, v_\zeta)\|^2 \nonumber\\
&\quad+ \frac{4}{\lambda_{\zeta}^2}\mathbb{E}  \| z_{\zeta} - x_{\zeta} + \mathrm{prox}_{\lambda_{\zeta}h}(x_\zeta - \lambda_{\zeta}v_\zeta) - \mathrm{prox}_{\lambda_{\zeta}h}(z_\zeta - \lambda_{\zeta}v_\zeta)\|^2   \nonumber\\
&\le 2\mathbb{E}  \|\nabla f(z_{\zeta}) - v_{\zeta}\|^2 + 4\mathbb{E}\|G_{\lambda_{\zeta}}(x_{\zeta}, v_\zeta)\|^2 \nonumber\\
&\quad+ \frac{8}{\lambda_{\zeta}^2}\mathbb{E}  \| z_{\zeta} - x_{\zeta}\|^2 +\frac{8}{\lambda_{\zeta}^2}\mathbb{E}  \|\mathrm{prox}_{\lambda_{\zeta}h}(x_\zeta - \lambda_{\zeta}v_\zeta) - \mathrm{prox}_{\lambda_{\zeta}h}(z_\zeta - \lambda_{\zeta}v_\zeta)\|^2   \nonumber\\
&\overset{(ii)}{\le} 2\mathbb{E}  \|\nabla f(z_{\zeta}) - v_{\zeta}\|^2 + 4\mathbb{E}\|G_{\lambda_{\zeta}}(x_{\zeta}, v_\zeta)\|^2 + \frac{8}{\lambda_{\zeta}^2}\mathbb{E}  \| z_{\zeta} - x_{\zeta}\|^2 +\frac{8}{\lambda_{\zeta}^2}\mathbb{E}  \|z_{\zeta} - x_{\zeta}\|^2   \nonumber\\
&\overset{(iii)}{\le} 2\mathbb{E}  \|\nabla f(z_{\zeta}) - v_{\zeta}\|^2 + 4\mathbb{E}\|G_{\lambda_{\zeta}}(x_{\zeta}, v_\zeta)\|^2 + \frac{16}{\lambda_{\zeta}^2}\mathbb{E}  \| y_{\zeta} - x_{\zeta}\|^2 \label{eq: 29}
\end{align}
where (i) uses the non-expansiveness property of the operator $G$ in \Cref{aux: 4}, (ii) follows from the non-expansiveness of the proximal operator, and (iii) uses the update rule and the fact that $0<\alpha_k<1$.

Next, we bound the three terms on the right hand side of the above inequality separately. First, note that
\begin{align*}
\mathbb{E} \|G_{\lambda_{\zeta}}(x_{\zeta}, v_\zeta)\|^2  = \frac{1}{K }\sum_{k=0}^{K-1} \mathbb{E} \|G_{\lambda_{k}}(x_{k}, v_k)\|^2 \le \frac{16(\Psi(x_0) - \Psi^*)}{K\beta}. 
\end{align*}
Second, note that \cref{eq: 17} implies that
\begin{align*}
& \mathbb{E} \|\nabla f(z_\zeta)-v_\zeta\|^2  \nonumber \\
&\le \mathbb{E}\sum_{i=(\tau(\zeta)-1)q}^{\zeta-1}\frac{L^2}{|\xi_i|} \big[2\beta_i^2\|G_{\lambda_{i}}(x_{i}, v_i)\|^2 + 2\alpha_{i+2}^2 \Gamma_{i+1}\sum_{t=0}^{i} \frac{(\lambda_{t} - \beta_{t})^2}{\alpha_{t+1}\Gamma_{t+1}} \|G_{\lambda_{t}}(x_{t}, v_t)\|^2 \big] \nonumber\\
&\le \frac{2L^2\beta^2}{|\xi|} \mathbb{E} \bigg(\sum_{i=(\tau(\zeta)-1)q}^{\tau(\zeta)q-1} \|G_{\lambda_{i}}(x_{i}, v_i)\|^2\bigg) + \frac{L^2}{|\xi|} \mathbb{E} \bigg(\sum_{i=(\tau(\zeta)-1)q}^{\zeta-1} 2\alpha_{i+2}^2 \Gamma_{i+1}\sum_{t=0}^{i} \frac{(\lambda_{t} - \beta_{t})^2}{\alpha_{t+1}\Gamma_{t+1}} \|G_{\lambda_{t}}(x_{t}, v_t)\|^2 \bigg) \nonumber\\
&\le \frac{2L^2\beta^2}{|\xi|K} \sum_{\zeta=0}^{K-1} \bigg(\sum_{i=(\tau(\zeta)-1)q}^{\tau(\zeta)q-1} \mathbb{E}\|G_{\lambda_{i}}(x_{i}, v_i)\|^2\bigg)+ \frac{L^2 \beta^2}{|\xi|K}  \sum_{\zeta=0}^{K-1} \bigg(\sum_{i=(\tau(\zeta)-1)q}^{\tau(\zeta)q-1} 2\alpha_{i+2}^2 \Gamma_{i+1}\sum_{t=0}^{\zeta-1} (t+1) \mathbb{E}\|G_{\lambda_{t}}(x_{t}, v_t)\|^2 \bigg) \nonumber\\
&\le \frac{2L^2\beta^2 q}{|\xi|} \frac{1}{K} \sum_{\zeta=0}^{K-1} \mathbb{E}\|G_{\lambda_{\zeta}}(x_{\zeta}, v_\zeta)\|^2 + \frac{L^2 \beta^2}{|\xi|} \frac{1}{K} \sum_{\zeta=0}^{K-1} \bigg(\frac{4}{\tau(\zeta)^3}\sum_{t=0}^{\zeta-1} (\ceil[]{t/q}+1) \mathbb{E}\|G_{\lambda_{t}}(x_{t}, v_t)\|^2 \bigg) \nonumber\\
&\le 2L^2\beta^2 \bigg(\frac{1}{K} \sum_{\zeta=0}^{K-1} \mathbb{E}\|G_{\lambda_{\zeta}}(x_{\zeta}, v_\zeta)\|^2\bigg) + \frac{L^2 \beta^2}{|\xi|} \frac{1}{K} \sum_{\zeta=0}^{K-1} (\ceil[]{\zeta/q}+1) \mathbb{E}\|G_{\lambda_{\zeta}}(x_{\zeta}, v_\zeta)\|^2 \sum_{t=\zeta}^{K-1} \frac{4}{\tau(t)^3} \nonumber\\
&\le 2L^2\beta^2 \bigg(\frac{1}{K} \sum_{\zeta=0}^{K-1} \mathbb{E}\|G_{\lambda_{\zeta}}(x_{\zeta}, v_\zeta)\|^2\bigg) + L^2 \beta^2 \frac{1}{K} \sum_{\zeta=0}^{K-1} \mathbb{E}\|G_{\lambda_{\zeta}}(x_{\zeta}, v_\zeta)\|^2 \frac{2(\ceil[]{\zeta/q}+1)}{\tau(\zeta)]^2} \nonumber\\
&\le 3L^2\beta \frac{16(\Psi(x_0) - \Psi^*)}{K}, \nonumber
\end{align*} 
where we have used the fact that $\zeta$ is sampled uniformly from $0,...,K-1$ at random. 

Third, note that by item 2 of \Cref{aux: 3}, we know that
\begin{align}
	\mathbb{E}  \| y_{\zeta} - x_{\zeta}\|^2 &\le \mathbb{E} \bigg(\Gamma_\zeta\sum_{t=0}^{\zeta-1} \frac{\lambda_{t} - \beta_{t}}{\alpha_{t+1}\Gamma_{t+1}}  \|G_{\lambda_{t}}(x_{t}, v_{t})\|^2\bigg) \nonumber\\
	&\le \frac{1}{K} \sum_{\zeta=0}^{K-1} \Gamma_\zeta\sum_{t=0}^{\zeta-1} \frac{\lambda_{t} - \beta_{t}}{\alpha_{t+1}\Gamma_{t+1}} \mathbb{E}  \|G_{\lambda_{t}}(x_{t}, v_{t})\|^2 \nonumber\\
	&\le \frac{\beta^2}{K} \sum_{\zeta=0}^{K-1} \Gamma_\zeta\sum_{t=0}^{\zeta-1} (\ceil[]{t/q}+1) \mathbb{E}  \|G_{\lambda_{t}}(x_{t}, v_{t})\|^2 \nonumber\\
	&= \frac{\beta^2}{K} \sum_{\zeta=0}^{K-1} (\ceil[]{\zeta/q}+1) \mathbb{E}  \|G_{\lambda_{\zeta}}(x_{\zeta}, v_{\zeta})\|^2 \bigg(\sum_{t=\zeta+1}^{K-1}\Gamma_t \bigg) \nonumber\\
	&\le \frac{\beta^2}{K} \sum_{\zeta=0}^{K-1} \mathbb{E}  \|G_{\lambda_{\zeta}}(x_{\zeta}, v_{\zeta})\|^2.
\end{align}

Combining the above three inequalities and note that $L\beta = \Theta(1)$ and $\frac{\beta}{\lambda_k} \le 1$, we finally obtain that
\begin{align}
\mathbb{E}  \|G_{\lambda_{\zeta}}(z_{\zeta}, \nabla f(z_{\zeta}))\|^2  \le \mathcal{O}\bigg(\frac{L(\Psi(x_0) - \Psi^*)}{K} \bigg) . \label{eq: 20}
\end{align}
This further implies that
\begin{align*}
\mathbb{E}  \|G_{\lambda_{\zeta}}(z_{\zeta}, \nabla f(z_{\zeta}))\| \le \sqrt{\mathbb{E}  \|G_{\lambda_{\zeta}}(z_{\zeta}, \nabla f(z_{\zeta}))\|^2} \le \mathcal{O}\bigg(\sqrt{\frac{L(\Psi(x_0) - \Psi^*)}{K}}\bigg). 
\end{align*}
Setting the right hand side of the above inequality to be bounded by $\epsilon$, we obtain that $K \ge \mathcal{O}\bigg(\frac{L(\Psi(x_0) - \Psi^*)}{\epsilon^2} \bigg)$. Then, the total number of stochastic gradient calls is bounded by $(K+q)\frac{n}{q} + K|\xi| \le \mathcal{O}(n+\sqrt{n}\epsilon^{-2})$.

\subsection{Proof of \Cref{thm: ProxSpiderO}}

The proof follows exactly from that of \Cref{thm: ProxSpiderM} (the same treatment of the momentum schemes applies).

\end{document}